\documentclass[10pt]{amsart}

\usepackage{tikz,tikz-qtree,ifthen}

\usepackage{amsfonts}
\usepackage{amsmath}
\usepackage{amssymb}
\usepackage{latexsym}

\usepackage{array,tikz-qtree,ifthen,cancel}
\usepackage{graphicx}
\usepackage{adjustbox}
\usepackage{amsthm,graphicx,tikz,appendix,tikz-qtree,ifthen,cancel}
\usetikzlibrary{calc,shapes,patterns,positioning}
\tikzset{hide on/.code={\only<#1>{\color{fg!0}}}}

\usepgflibrary{shapes.geometric} % LATEX and plain TEX and pure pgf
\usepgflibrary[shapes.geometric] % ConTEXt and pure pgf
\usetikzlibrary{shapes.geometric} % LATEX and plain TEX when using Tik Z
\tikzset{hide on/.code={\only<#1>{\color{fg!0}}}}
\usepackage{tikz,amsmath,amsthm,amssymb}

\newtheorem{thm}{Theorem}%[section]
\newtheorem{lem}[thm]{Lemma}

\newtheorem{claim}{Claim}
\newtheorem{fact}[thm]{Fact}

\newtheorem{problem}[thm]{Problem}
\newtheorem{question}[thm]{Question}

\theoremstyle{remark}
\newtheorem*{rem}{Remark}

\theoremstyle{definition}
\newtheorem{defn}[thm]{Definition}

\theoremstyle{remark}

\newcommand{\al}{\alpha}
\newcommand{\om}{\omega}

\newcommand{\sse}{\subseteq}
\newcommand{\contains}{\supseteq}
\newcommand{\forces}{\Vdash}

\DeclareMathOperator{\dom}{dom}

\DeclareMathOperator{\stem}{stem}

\DeclareMathOperator{\spl}{spl}

\newcommand{\re}{\restriction}
\newcommand{\bP}{\mathbb{P}}

\newcommand{\bD}{\mathbb{D}}

\newcommand{\bQ}{\mathbb{Q}}

\newcommand{\bT}{\mathbb{T}}
\newcommand{\bS}{\mathbb{S}}

\newcommand{\G}{\mathrm{G}}

\newcommand{\F}{\mathrm{F}}
\newcommand{\R}{\mathrm{R}}

\newcommand{\ra}{\rightarrow}

\newcommand{\lgl}{\langle}
\newcommand{\rgl}{\rangle}

\newcommand{\Nesetril}{Ne{\v{s}}et{\v{r}}il}

\newcommand{\Rodl}{R{\"{o}}dl}
\newcommand{\Erdos}{Erd{\H{o}}s}

\newcommand{\Fraisse}{Fra{\"{i}}ss{\'{e}}}
\newcommand{\Lauchli}{L{\"{a}}uchli}
\newcommand{\Dzamonja}{D{\v{z}}amonja}

\newcommand{\noprint}[1]{\relax}

\title[Forcing in Ramsey Theory]{Forcing in Ramsey Theory}

\author{Natasha Dobrinen}
\address{University of Denver\\
Department of Mathematics, 2280 S Vine St, Denver, CO 80208, USA}
\email{natasha.dobrinen@du.edu}
\thanks{The author was partially supported by  National Science Foundation Grant DMS-1600781}


\begin{document}

\maketitle

\begin{abstract}
Ramsey theory and forcing have a symbiotic relationship.
At the RIMS Symposium on Infinite Combinatorics and Forcing Theory in 2016, the author gave three tutorials on {\em Ramsey theory in forcing}.
The
first two tutorials concentrated on forcings which contain dense subsets forming topological Ramsey spaces.
These forcings  motivated the development of new Ramsey theory, which then was applied  to the generic ultrafilters
to obtain the precise structure  Rudin-Keisler and Tukey orders below such ultrafilters.
The content of the first two tutorials has appeared in a previous paper \cite{DobrinenSEALS17}.
The third tutorial concentrated on uses  of forcing to prove Ramsey theorems for trees which are applied to determine big Ramsey degrees of homogeneous relational structures.
This is the focus of this paper.

\end{abstract}

%%%%%%%%%%%%%%%
%%%%%%%%%%%%%%%
%%%%%%%%%%%%%%%
%%%%%%%%%%%%%%%

\section{Overview of Tutorial}

Ramsey theory and forcing are deeply interconnected in a multitude of various ways.
At the RIMS Conference on Infinite Combinatorics and Forcing Theory,
the author gave a series of three tutorials on {\em Ramsey theory in forcing}.
These tutorials focused  on  the following implications between forcing and Ramsey theory:
\begin{enumerate}
\item
Forcings adding ultrafilters satisfying weak partition relations, which in turn motivate new topological Ramsey spaces and canonical equivalence relations on barriers;
\item
Applications  of Ramsey theory to
obtain the  precise  Rudin-Keisler and Tukey structures below these forced ultrafilters;
\item
Ramsey theory motivating new forcings and associated ultrafilters;
\item
Forcing to obtain new Ramsey theorems for trees;
\item
Applications of new Ramsey theorems for trees to obtain Ramsey theorems for homogeneous relational structures.
\end{enumerate}

The first two tutorials focused on areas (1) - (3).
We presented work from
 \cite{Dobrinen/Todorcevic14}, \cite{Dobrinen/Todorcevic15} \cite{Dobrinen/Mijares/Trujillo14},
 \cite{DobrinenJSL15},
\cite{DobrinenJML16},
and \cite{DobrinenCreaturetRs15},
in which
dense subsets of  forcings   generating  ultrafilters satisfying some weak partition properties
were shown to
 form topological Ramsey spaces.
Having obtained
the canonical  equivalence relations on fronts  for these topological Ramsey spaces,
they may be applied to obtain the precise
initial
Rudin-Keisler and Tukey  structures.
An exposition of this work has already appeared in \cite{DobrinenSEALS17}.

The third tutorial concentrated on
areas (4) and (5).
We focused  particularly on the Halpern-\Lauchli\ Theorem and variations for strong trees.
Extensions and analogues of this theorem have found applications in homogeneous relational structures.
The majority of this article will concentrate
on  Ramsey theorems for trees due to (in historical order)  Halpern-\Lauchli\ \cite{Halpern/Lauchli66}; Milliken \cite{Milliken81};   Shelah \cite{Shelah91};
\Dzamonja, Larson, and Mitchell \cite{Dzamonja/Larson/MitchellQ09}; Dobrinen and Hathaway \cite{Dobrinen/Hathaway16}; Dobrinen \cite{DobrinenH_317}; and very recently Zhang \cite{Zhang17}.
These theorems have important applications to  finding
big Ramsey degrees for homogeneous structures.

We say that an infinite structure $\mathcal{S}$ has {\em finite big Ramsey degrees} if for each finite substructure $\F$ of $\mathcal{S}$
there is some finite number $n(\F)$ such that for any coloring of all copies of $\F$ in $\mathcal{S}$ into finitely many colors, there is a substructure $\mathcal{S}'$ if $\mathcal{S}$ which is isomorphic to $\mathcal{S}$ and such that all copies of $\F$ in $\mathcal{S}'$ take no more than $n(\F)$ colors.
The Halpern-\Lauchli\ and Milliken Theorems, and other related Ramsey theorems on trees, have been instrumental in  proving  finite big Ramsey degrees for certain homogeneous relational structures.
Section \ref{sec.HL}
contains  Harrington's  forcing proof of the Halpern-\Lauchli\ Theorem.
This is then applied to obtain Milliken's Theorem for strong subtrees.
Applications of Milliken's theorem to
obtain finite big Ramsey degrees are shown in Section \ref{sec.Sauer}.
There, we provide the
main ideas of how Sauer applied Milliken's Theorem to prove that
the random graph on countably many vertices
has finite big Ramsey degrees in \cite{Sauer06}.
Then we briefly cover applications to Devlin's work in \cite{DevlinThesis} on
finite subsets of the rationals,
 and Laver's work in \cite{Laver84} on
finite products of the rationals.

In another vein, proving whether or not homogeneous relational structures omitting copies of a certain finite structure have finite big Ramsey degrees has been an elusive endeavor until recently.
In \cite{DobrinenH_317}, the author used forcing to prove the needed analogues of Milliken's Theorem and applied them to prove  the universal triangle-free graph
has finite big Ramsey degrees.
The main ideas in that paper are covered in Section \ref{sec.H_3}.

The final Section \ref{sec.measurable}
addresses analogues of the Halpern-\Lauchli\ Theorem for trees of uncountable height.
The first such theorem, due to Shelah \cite{Shelah91} (strengthened in \cite{Dzamonja/Larson/MitchellQ09})
considers finite antichains in one tree of measurable height.
This
was applied by
\Dzamonja, Larson, and Mitchell  to prove the consistency of a measurable cardinal $\kappa$
and the analogue of
Devlin's result for the
$\kappa$-rationals in \cite{Dzamonja/Larson/MitchellQ09};
and  that the $\kappa$-Rado graph
 has finite big Ramsey degrees in \cite{Dzamonja/Larson/MitchellRado09}.
Recent work of Hathaway and the author in \cite{Dobrinen/Hathaway16} considers more than one tree and implications for  various uncountable cardinals.
We conclude the paper with very recent results of  Zhang in \cite{Zhang17} obtaining the analogue of Laver's result at a measurable cardinal.

The existence of finite big Ramsey degrees has been of interest for some time to those studying homogeneous structures.
In addition to those results considered in this paper,
big Ramsey degrees have been investigated in the context of ultrametric spaces in \cite{NVT08}.
A recent connection between finite big Ramsey degrees and topological dynamics
 has been made by Zucker in \cite{Zucker17}.
Any future  progress on finite big Ramsey degrees  will have implications for topological dynamics.

The author would like to thank Timothy Trujillo for creating most of the diagrams used in the tutorials, and the most complex ones included here.

%%%%%%%%%%%%%%%
%%%%%%%%%%%%%%%
%%%%%%%%%%%%%%%
%%%%%%%%%%%%%%%

%%%%%%%%%%%%%%%
%%%%%%%%%%%%%%%%
%%%%%%%%%%%%%%%%%%
%%%%%%%%%%%%%%%%%

\section{The Halpern-\Lauchli\ and Milliken Theorems}\label{sec.HL}

Ramsey theory on trees is a powerful tool for investigations into several branches of mathematics.
The Halpern-\Lauchli\ Theorem was originally proved as a main technical lemma enabling a later  proof
of Halpern and Lev\'{y}
 that the Boolean prime ideal theorem is strictly weaker than the Axiom of Choice, assuming the ZF axioms (see \cite{Halpern/Levy71}).
Many variations of this theorem have been proved.
We will
 concentrate on
 the strong tree version of the Halpern-\Lauchli\ Theorem, referring the interested reader   to Chapter 3 in \cite{TodorcevicBK10} for a compendium of other variants.
An extension due to  Milliken, which is a Ramsey theorem on strong trees, has found numerous applications finding precise structural properties  of homogeneous structures, such as the Rado graph and the rational numbers, in terms of Ramsey degrees for colorings of finite substructures.
This will be presented in the second part of this section.

Several proofs of the Halpern-\Lauchli\ Theorem are available in the literature.
A proof using the technique of forcing was discovered by Harrington, and is regarded as providing the most insight.
It was  known to a handful of  set theorists for several decades.
His proof uses a cardinal which satisfies the following  partition relation for colorings of
subsets of size $2d$
 into countably many colors, and uses Cohen forcing to add  many paths through each of the trees.

\begin{defn}\label{defn.arrownotation}
Given cardinals $d,\sigma,\kappa,\lambda$,
\begin{equation}
\lambda\ra(\kappa)^d_{\sigma}
\end{equation}
means that for each coloring of $[\lambda]^d$ into $\sigma$ many colors,
there is a subset $X$ of $\lambda$ such that $|X|=\kappa$ and all members of $[X]^d$ have the same color.
\end{defn}

The following is  a ZFC result guaranteeing cardinals large enough to have the Ramsey property for colorings into infinitely many colors.

\begin{thm}[\Erdos-Rado]\label{thm.ER}
For $r<\om$ and $\mu$ an infinite cardinal,
$$
\beth_r(\mu)^+\ra(\mu^+)_{\mu}^{r+1}.
$$
\end{thm}

The book
\cite{Farah/TodorcevicBK} of Farah and Todorcevic contains a forcing proof of the Halpern-\Lauchli\ Theorem.
The proof there is a
modified version of Harrington's original proof.
It uses a cardinal satisfying the weaker partition
 relation $\kappa\ra(\aleph_0)^d_{2}$,  which is satisfied by the cardinal
$\beth^+_{d-1}$.
This  is important if one is interested in obtaining the theorem  from weaker assumptions, and was
instrumental
in motivating  the main result in
  \cite{Dobrinen/Hathaway16} which is presented in Section \ref{sec.measurable}.
One could conceivably recover Harrington's original argument from Shelah's proof of the Halpern-\Lauchli\ Theorem at a measurable cardinal (see \cite{Shelah91}).
However, his proof is more complex than simply  lifting  Harrington's argument to a measurable cardinal, as he obtiains a stronger version, but only for one tree (see Section \ref{sec.measurable}).
Thus, we present here  the simplest version of Harrington's forcing proof, filling a hole in the literature at present.
This version was outlined to the author in 2011 by Richard Laver, and the author has filled in the gaps.

A {\em tree} on $\om^{<\om}$ is
a subset $T\sse \om^{\om}$  which is closed under meets.
Thus, in this article,  a tree is not necessarily closed  under initial segments.
We let
$\widehat{T}$ denote the set of all initial segments of members of $T$;
thus, $\widehat{T}=\{s\in\om^{<\om}:\exists t\in T(s\sse t)\}$.
Given any tree $T\sse\om^{<\om}$ and a node $t\in T$,
let $\spl_T(t)$ denote the set of all immediate successors of $t$ in $\widehat{T}$;
thus, $\spl_T(t)=\{u\in \widehat{T}: u\contains t$ and $|u|=|t|+1\}$.
Notice that the nodes in $\spl_T(t)$ are not necessarily nodes in $T$.
For a tree $T\sse 2^{<\om}$ and $n<\om$, let $T(n)$ denote $T\cap 2^n$; thus, $T(n)=\{t\in T:|t|=n\}$.
A set $X\sse T$ is a {\em level set} if all nodes in $X$ have the same length.
Thus, $X\sse T$ is a level set if $X\sse T(n)$ for some  $n<\om$.

Let $T\sse \om^{<\om}$ be a finitely branching tree with no terminal nodes such  that $\hat{T}=T$, and each node in $T$ splits into at least two immediate successors.
A subtree $S\sse T$ is an infinite
{\em strong subtree of $T$}  if there is an infinite set $L\sse\om$  of levels such that
\begin{enumerate}
\item
$S=\bigcup_{l\in L}\{s\in S:|s|=l\}$;
\item
for each node $s\in S$,
$s$ splits in $S$ if and only if $|s|\in L$;
\item
if $|s|\in L$, then  $\spl_S(s)=\spl_T(s)$.
\end{enumerate}
$S$ is a {\em finite strong subtree} of $T$ if there is a finite set of levels  $L$ such that (1) holds, and every non-maximal node in $S$ at a level in $L$ splits maximally in $T$.
See Figures 1 and 2 for examples of finite strong trees isomorphic to $2^{\le 2}$, as determined by the darkened nodes.

%%%%%%%%%%%%%%%%%%%%%%%%%%%%
%%%%%%%%%%%%%%%%%%%%%%%%%%%

\begin{figure}
\begin{tikzpicture}[grow'=up,scale=.6]
\tikzstyle{level 1}=[sibling distance=4in]
\tikzstyle{level 2}=[sibling distance=2in]
\tikzstyle{level 3}=[sibling distance=1in]
\tikzstyle{level 4}=[sibling distance=0.5in]
\tikzstyle{level 5}=[sibling distance=0.2in]
\node {} coordinate (t9)
child{coordinate (t0)
			child{coordinate (t00)
child{coordinate (t000)
child {coordinate(t0000)
child{coordinate(t00000) edge from parent[color=black] }
child{coordinate(t00001)}}
child {coordinate(t0001) edge from parent[color=black]
child {coordinate(t00010)}
child{coordinate(t00011)}}}
child{ coordinate(t001)
child{ coordinate(t0010)
child{ coordinate(t00100)}
child{ coordinate(t00101) edge from parent[color=black] }}
child{ coordinate(t0011) edge from parent[color=black]
child{ coordinate(t00110)}
child{ coordinate(t00111)}}}}
			child{ coordinate(t01)  edge from parent[color=black]
child{ coordinate(t010)
child{ coordinate(t0100)
child{ coordinate(t01000)}
child{ coordinate(t01001)}}
child{ coordinate(t0101)
child{ coordinate(t01010)}
child{ coordinate(t01011)}}}
child{ coordinate(t011)
child{ coordinate(t0110)
child{ coordinate(t01100)}
child{ coordinate(t01101)}}
child{ coordinate(t0111)
child { coordinate(t01110)}
child{ coordinate(t01111)}}}}}
		child{ coordinate(t1)
			child{ coordinate(t10)
child{ coordinate(t100)
child{ coordinate(t1000) edge from parent[color=black]
child{ coordinate(t10000)}
child{ coordinate(t10001)}}
child{ coordinate(t1001)
child{ coordinate(t10010)}
child{ coordinate(t10011) edge from parent[color=black] }}}
child{ coordinate(t101)
child{ coordinate(t1010) edge from parent[color=black]
child{ coordinate(t10100) }
child{ coordinate(t10101)}}
child{ coordinate(t1011)
child{ coordinate(t10110) edge from parent[color=black] }
child{ coordinate(t10111)}}}}
			child{  coordinate(t11)  edge from parent[color=black]
child{ coordinate(t110)
child{ coordinate(t1100)
child{ coordinate(t11000)}
child{ coordinate(t11001)}}
child{ coordinate(t1101)
child{ coordinate(t11010)}
child{ coordinate(t11011)}}}
child{  coordinate(t111)
child{  coordinate(t1110)
child{  coordinate(t11100)}
child{  coordinate(t11101)}}
child{  coordinate(t1111)
child{  coordinate(t11110)}
child{  coordinate(t11111)}}}} };

\node[left] at (t0) {$0$};
\node[left] at (t00) {$00$};
\node[left] at (t000) {$000$};
\node[left] at (t001) {$001$};
\node[left] at (t01) {$01$};
\node[left] at (t010) {$010$};
\node[left] at (t011) {$011$};
\node[right] at (t1) {$1$};
\node[right] at (t10) {$10$};
\node[right] at (t100) {$100$};
\node[right] at (t101) {$101$};
\node[right] at (t11) {$11$};
\node[right] at (t110) {$110$};
\node[right] at (t111) {$111$};

\node[circle, fill=black,inner sep=0pt, minimum size=6pt] at (t9) {};

\node[circle, fill=black,inner sep=0pt, minimum size=6pt] at (t00) {};
\node[circle, fill=black,inner sep=0pt, minimum size=6pt] at (t00001) {};
\node[circle, fill=black,inner sep=0pt, minimum size=6pt] at (t00100) {};
\node[circle, fill=black,inner sep=0pt, minimum size=6pt] at (t10010) {};
\node[circle, fill=black,inner sep=0pt, minimum size=6pt] at (t10) {};
\node[circle, fill=black,inner sep=0pt, minimum size=6pt] at (t10111) {};

\end{tikzpicture}
\caption{A strong tree isomorphic to $2^{\le 2}$}
\end{figure}
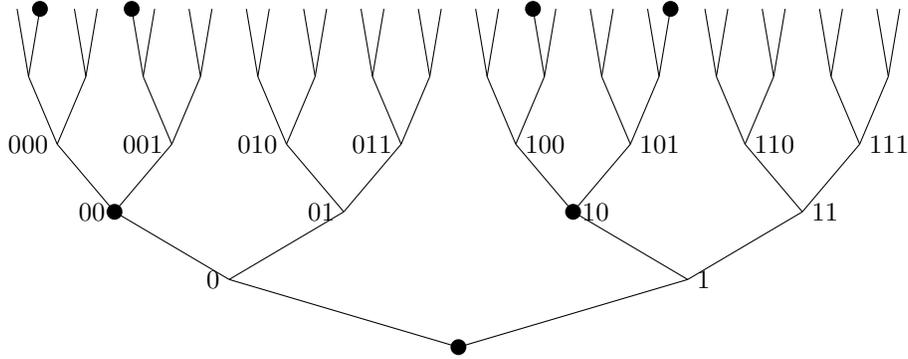

%%%%%%%%%%%%%%%%%%%%%%%%%%%%
%%%%%%%%%%%%%%%%%%%%%%%%%%%

\begin{figure}
\begin{tikzpicture}[grow'=up,scale=.6]
\tikzstyle{level 1}=[sibling distance=4in]
\tikzstyle{level 2}=[sibling distance=2in]
\tikzstyle{level 3}=[sibling distance=1in]
\tikzstyle{level 4}=[sibling distance=0.5in]
\tikzstyle{level 5}=[sibling distance=0.2in]
\node {} coordinate (t9)
child{coordinate (t0)
			child{coordinate (t00)
child{coordinate (t000) edge from parent[color=black]
child {coordinate(t0000)
child{coordinate(t00000)}
child{coordinate(t00001)}}
child {coordinate(t0001)
child {coordinate(t00010)}
child{coordinate(t00011)}}}
child{ coordinate(t001)
child{ coordinate(t0010)
child{ coordinate(t00100)}
child{ coordinate(t00101) edge from parent[color=black] }}
child{ coordinate(t0011)
child{ coordinate(t00110)}
child{ coordinate(t00111)edge from parent[color=black] }}}}
			child{ coordinate(t01)
child{ coordinate(t010) edge from parent[color=black]
child{ coordinate(t0100)
child{ coordinate(t01000)}
child{ coordinate(t01001)}}
child{ coordinate(t0101)
child{ coordinate(t01010)}
child{ coordinate(t01011)}}}
child{ coordinate(t011)
child{ coordinate(t0110)
child{ coordinate(t01100)}
child{ coordinate(t01101)edge from parent[color=black] }}
child{ coordinate(t0111)
child { coordinate(t01110) edge from parent[color=black] }
child{ coordinate(t01111)}}}}}
		child{ coordinate(t1)
			child{ coordinate(t10)
child{ coordinate(t100)
child{ coordinate(t1000)
child{ coordinate(t10000)}
child{ coordinate(t10001)}}
child{ coordinate(t1001)
child{ coordinate(t10010)}
child{ coordinate(t10011)}}}
child{ coordinate(t101)
child{ coordinate(t1010)
child{ coordinate(t10100)}
child{ coordinate(t10101)}}
child{ coordinate(t1011)
child{ coordinate(t10110)}
child{ coordinate(t10111)}}}}
			child{  coordinate(t11)
child{ coordinate(t110)
child{ coordinate(t1100)
child{ coordinate(t11000)}
child{ coordinate(t11001)}}
child{ coordinate(t1101)
child{ coordinate(t11010)}
child{ coordinate(t11011)}}}
child{  coordinate(t111)
child{  coordinate(t1110)
child{  coordinate(t11100)}
child{  coordinate(t11101)}}
child{  coordinate(t1111)
child{  coordinate(t11110)}
child{  coordinate(t11111)}}}} };

\node[left] at (t0) {$0$};
\node[left] at (t00) {$00$};
\node[left] at (t000) {$000$};
\node[left] at (t001) {$001$};
\node[left] at (t01) {$01$};
\node[left] at (t010) {$010$};
\node[left] at (t011) {$011$};
\node[right] at (t1) {$1$};
\node[right] at (t10) {$10$};
\node[right] at (t100) {$100$};
\node[right] at (t101) {$101$};
\node[right] at (t11) {$11$};
\node[right] at (t110) {$110$};
\node[right] at (t111) {$111$};

\node[circle, fill=black,inner sep=0pt, minimum size=6pt] at (t0) {};
\node[circle, fill=black,inner sep=0pt, minimum size=6pt] at (t001) {};
\node[circle, fill=black,inner sep=0pt, minimum size=6pt] at (t00100) {};
\node[circle, fill=black,inner sep=0pt, minimum size=6pt] at (t00110) {};
\node[circle, fill=black,inner sep=0pt, minimum size=6pt] at (t011) {};
\node[circle, fill=black,inner sep=0pt, minimum size=6pt] at (t01100) {};
\node[circle, fill=black,inner sep=0pt, minimum size=6pt] at (t01111) {};

\end{tikzpicture}
\caption{Another strong tree isomorphic to $2^{\le 2}$}
\end{figure}
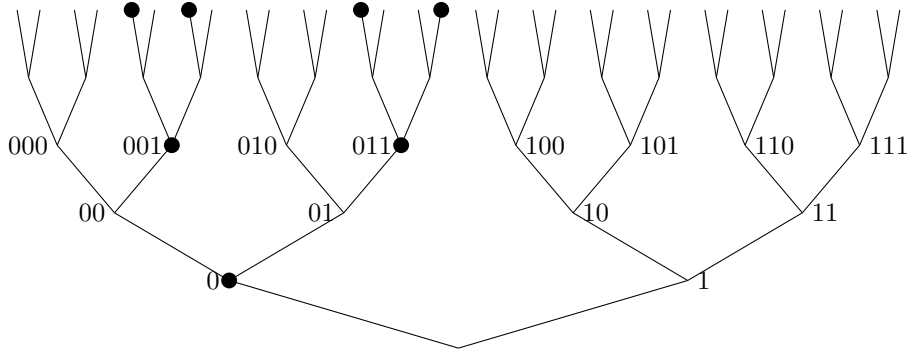

%%%%%%%%%%%%%%%%%%%%

We now present Harrington's proof of the Halpern-\Lauchli\ Theorem, as outlined by Laver   and filled in by the author.
Although the
  proof  uses the set-theoretic technique of forcing,
the whole  construction takes place  in the original model of ZFC, not some generic extension.
The forcing should be thought of as conducting an unbounded search for a finite object, namely the next level set  where homogeneity is attained.

\begin{thm}[Halpern-\Lauchli]\label{thm.matrixHL}
Let $1< d<\om$ and let $T_i\sse\om^{<\om}$ be finitely branching trees such that $\widehat{T}_i= T_i$.
Let
\begin{equation}
c:\bigcup_{n<\om}\prod_{i<d} T_i(n)\ra 2
\end{equation}
 be given.
Then there is an infinite set of levels $L\sse \om$ and strong subtrees $S_i\sse T_i$  each with branching nodes exactly at the levels in $L$
such that $c$ is monochromatic on
\begin{equation}
\bigcup_{n\in L}\prod_{i<d} S_i(n).
\end{equation}
\end{thm}

\begin{proof}
Let
$c:\bigcup_{n<\om}\prod_{i<d} T_i(n) \ra 2$
be given.
Let $\kappa=\beth_{2d}$.
The following forcing notion $\bP$  will  add $\kappa$ many paths through each $T_i$, $i\in d$.
$\bP$ is the set of conditions $p$ such that
$p$ is a  function
of the form
\begin{equation}
p: d\times\vec{\delta}_p\ra \bigcup_{i<d} T_i(l_p)
\end{equation}
where
\begin{enumerate}
\item[(i)]
$\vec{\delta}_p\in[\kappa]^{<\om}$;
\item[(ii)]
$l_p<\om$; and
\item [(iii)]
for each $i<d$,
 $\{p(i,\delta) : \delta\in  \vec{\delta}_p\}\sse  T_i( l_p)$.
\end{enumerate}
The partial  ordering on $\bP$ is simply inclusion:
$q\le p$ if and only if
$l_q\ge l_p$, $\vec\delta_q\contains \vec\delta_p$,
and for each $(i,\delta)\in d\times \vec\delta_p$,
$q(i,\delta)\contains p(i,\delta)$.

$\bP$ adds $\kappa$ branches through each of tree $T_i$, $i<d$.
For each $i<d$ and $\al<\kappa$,
let  $\dot{b}_{i,\al}$ denote the $\al$-th generic branch through $T_i$.
Thus,
\begin{equation}
\dot{b}_{i,\al}=\{\lgl p(i,\al),p\rgl :p\in \bP,\ \mathrm{and}\ (i,\al)\in\dom(p)\}.
\end{equation}
Note that for each $p\in \bP$ with $(i,\al)\in\dom(p)$, $p$ forces that $\dot{b}_{i,\al}\re l_p= p(i,\al)$.

\begin{rem}
 $(\bP,\le)$ is just a fancy form of adding $\kappa$ many Cohen reals, where
 we add $d\times\kappa$ many Cohen reals, and the conditions are
 homogenized over the levels of trees in the ranges of conditions and over the finite set of ordinals indexing the generic branches.
\end{rem}

Let $\dot{\mathcal{U}}$ be a $\bP$-name for a non-principal ultrafilter on $\om$.
To ease  notation, we  shall write sets $\{\al_i:i< d\}$ in
$[\kappa]^d$ as vectors $\vec{\al}=\lgl \al_0,\dots,\al_{d-1}\rgl$ in strictly increasing order.
For $\vec{\al}=\lgl\al_0,\dots,\al_{d-1}\rgl\in[\kappa]^d$,
rather than writing out
$\lgl \dot{b}_{0,\al_0},\dots, \dot{b}_{d-1,\al_{d-1}}\rgl$
 each time we wish to refer to these generic branches,
we shall simply
\begin{equation}
\mathrm{let\ \ }\dot{b}_{\vec{\al}}\mathrm{\  \  denote\ \ }
\lgl \dot{b}_{0,\al_0},\dots, \dot{b}_{d-1,\al_{d-1}}\rgl.
\end{equation}
For any $l<\om$,
\begin{equation}
\mathrm{\ let\ \ }\dot{b}_{\vec\al}\re l
\mathrm{\ \ denote \  \ }
\{\dot{b}_{i,\al_i}\re l:i<d\}.
\end{equation}

The goal now is to find
 infinite  pairwise disjoint sets $K'_i\sse \kappa$, $i<d$,
and a set of conditions $\{p_{\vec\al}:\vec\al\in \prod_{i<d}K_i'\}$ which are compatible,
have the same images in $T$,
and such that for some fixed $\varepsilon^*$,
for each $\vec\al\in\prod_{i<d}K_i'$,
$p_{\vec\al}$ forces
$c(\dot{b}_{\vec\al}\re l)=\varepsilon^*$  for  ultrafilter $\dot{\mathcal{U}}$ many  $l$.
Moreover, we will find nodes $t^*_i$, $i\le d$, such that for each $\vec\al\in\prod_{i<d}K_i'$,
$p_{\vec\al}(i,\al_i)=t^*_i$.
These will  serve as the basis  for  the  process of building the strong subtrees  $S_i\sse T_i$ on which $c$ is monochromatic.

For each $\vec\al\in[\kappa]^d$,
choose a condition $p_{\vec{\al}}\in\bP$ such that
\begin{enumerate}
\item
 $\vec{\al}\sse\vec{\delta}_{p_{\vec\al}}$;
\item
$p_{\vec{\al}}\forces$ ``There is an $\varepsilon\in 2$  such that
$c(\dot{b}_{\vec{\al}}\re l)=\varepsilon$
for $\dot{\mathcal{U}}$ many $l$";
\item
$p_{\vec{\al}}$ decides a value for $\varepsilon$, label it  $\varepsilon_{\vec{\al}}$; and
\item
$c(\{p_{\vec\al}(i,\al_i):i< d\})=\varepsilon_{\vec{\al}}$.
\end{enumerate}

Such conditions $p_{\vec\al}$ may be obtained as follows.
Given $\vec\al\in[\kappa]^d$,
take $p_{\vec\al}^1$ to be any condition such that $\vec\al\sse\vec{\delta}_{p_{\vec\al}^1}$.
Since $\bP$ forces $\dot{\mathcal{U}}$ to be an ultrafilter on $\om$, there is a condition
 $p_{\vec\al}^2\le p_{\vec\al}^1$ such that
$p_{\vec\al}^2$ forces that $c(\dot{b}_{\vec\al}\re l)$ is the same color for $\dot{\mathcal{U}}$ many $l$.
Furthermore,  there must be a stronger
condition deciding which of the colors
$c(\dot{b}_{\vec\al}\re l)$ takes on $\dot{\mathcal{U}}$ many levels $l$.
Let $p_{\vec\al}^3\le p_{\vec\al}^2$  be a condition which decides this  color, and let $\varepsilon_{\vec\al}$ denote  that color.
Finally, since $p_{\vec\al}^3$ forces that for  $\dot{\mathcal{U}}$ many $l$ the color
 $c(\dot{b}_{\vec\al}\re l)$
will equal $\varepsilon_{\vec{\al}}$,
there is some $p_{\vec\al}^4\le p_{\vec\al}^3$ which decides some level $l$ so that
$c(\dot{b}_{\vec\al}\re l)=\varepsilon_{\vec{\al}}$.
If $l_{p_{\vec\al}^4}<l$,
let $p_{\vec\al}$ be any member of $\bP$ such that
$p_{\vec\al}\le p_{\vec\al}^4$ and $l_{p_{\vec\al}}=l$.
If $l_{p_{\vec\al}^4}\ge l$,
let $p_{\vec\al}=\{((i,\delta), p_{\vec\al}^4(i,\delta)\re l):(i,\delta)\in d\times\vec\delta_{p_{\vec\al}^4}\}$,
the truncation of $p_{\vec\al}^4$ to images that have length $l$.
Then $p_{\vec\al}$ forces that $\dot{b}_{\vec\al}\re l=\{p_{\vec\al}(i,\al_i):i< d\}$, and hence
 $p_{\vec\al}$  forces that
$c(\{p_{\vec\al}(i,\al_i):i< d\})=\varepsilon_{\vec{\al}}$.

We are assuming $\kappa=\beth_{2d}$, which is at least  $\beth_{2d-1}(\aleph_0)^+$,  so  $\kappa\ra(\aleph_1)^{2d}_{\aleph_0}$ by Theorem \ref{thm.ER}.

Now we prepare  for an application of the \Erdos-Rado Theorem. Given two sets of ordinals $J,K$ we shall write $J<K$ if and only if every member of $J$ is less than every member of $K$.
Let $D_e=\{0,2,\dots,2d-2\}$ and  $D_o=\{1,3,\dots,2d-1\}$, the sets of  even and odd integers less than $2d$, respectively.
Let $\mathcal{I}$ denote the collection of all functions $\iota: 2d\ra 2d$ such that
\begin{equation}
\{\iota(0),\iota(1)\}<\{\iota(2),\iota(3)\}<\dots<\{\iota(2d-2),\iota(2d-1)\}.
\end{equation}
Thus, each $\iota$ codes two strictly increasing sequences $\iota\re D_e$ and $\iota\re D_o$, each of length $d$.
For $\vec{\theta}\in[\kappa]^{2d}$,
$\iota(\vec{\theta}\,)$ determines the pair of sequences of ordinals
\begin{equation}
 (\theta_{\iota(0)},\theta_{\iota(2)},\dots,\theta_{\iota(2d-2))}), (\theta_{\iota(1)},\theta_{\iota(3)},\dots,\theta_{\iota(2d-1)}),
\end{equation}
both of which are members of $[\kappa]^d$.
Denote these as $\iota_e(\vec\theta\,)$ and $\iota_o(\vec\theta\,)$, respectively.
To ease notation, let $\vec{\delta}_{\vec\al}$ denote
$\vec\delta_{p_{\vec\al}}$,
 $k_{\vec{\al}}$ denote $|\vec{\delta}_{\vec\al}|$,
and let $l_{\vec{\al}}$ denote  $l_{p_{\vec\al}}$.
Let $\lgl \delta_{\vec{\al}}(j):j<k_{\vec{\al}}\rgl$
denote the enumeration of $\vec{\delta}_{\vec\al}$
in increasing order.

Define a coloring  $f$ on $[\kappa]^{2d}$ into countably many colors as follows:
Given  $\vec\theta\in[\kappa]^{2d}$ and
 $\iota\in\mathcal{I}$, to reduce the number of subscripts,  letting
$\vec\al$ denote $\iota_e(\vec\theta\,)$ and $\vec\beta$ denote $\iota_o(\vec\theta\,)$,
define
\begin{align}\label{eq.fiotatheta}
f(\iota,\vec\theta\,)=& \,
\lgl \iota, \varepsilon_{\vec{\al}}, k_{\vec{\al}},
\lgl \lgl p_{\vec{\al}}(i,\delta_{\vec{\al}}(j)):j<k_{\vec{\al}}\rgl:i< d\rgl,\cr
& \lgl  \lgl i,j \rgl: i< d,\ j<k_{\vec{\al}},  \vec{\delta}_{\vec{\al}}(j)=\al_i \rgl,
\lgl \lgl j,k\rgl:j<k_{\vec{\al}},\ k<k_{\vec{\beta}},\ \delta_{\vec{\al}}(j)=\delta_{\vec{\beta}}(k)\rgl\rgl.
\end{align}
Let $f(\vec{\theta}\,)$ be the sequence $\lgl f(\iota,\vec\theta\,):\iota\in\mathcal{I}\rgl$, where $\mathcal{I}$ is given some fixed ordering.
Since the range of $f$ is countable,
applying the \Erdos-Rado Theorem,
we obtain a subset $K\sse\kappa$ of cardinality $\aleph_1$
which is homogeneous for $f$.
Take $K'\sse K$ such that between each two members of $K'$ there is a member of $K$ and $\min(K')>\min(K)$.
Take subsets $K'_i\sse K'$ such that  $K'_0<\dots<K'_{d-1}$
and   each $|K'_i|=\aleph_0$.

\begin{claim}\label{claim.onetypes}
There are $\varepsilon^*\in 2$, $k^*\in\om$,
and $ \lgl t_{i,j}: j<k^*\rgl$, $i< d$,
 such that
 $\varepsilon_{\vec{\al}}=\varepsilon^*$,
$k_{\vec\al}=k^*$,   and
$\lgl p_{\vec\al}(i,\delta_{\vec\al}(j)):j<k_{\vec\al}\rgl
=
 \lgl t_{i,j}: j<k^*\rgl$,
for each $i< d$,
for all $\vec{\al}\in \prod_{i<d}K'_i$.
\end{claim}

\begin{proof}
Let  $\iota$ be the member in $\mathcal{I}$
which is the identity function on $2d$.
For any pair $\vec{\al},\vec{\beta}\in \prod_{i<d}K'_i$, there are $\vec\theta,\vec\theta'\in [K]^{2d}$
such that
$\vec\al=\iota_e(\vec\theta\,)$ and $\vec\beta=\iota_e(\vec\theta'\,)$.
Since $f(\iota,\vec\theta\,)=f(\iota,\vec\theta'\,)$,
it follows that $\varepsilon_{\vec\al}=\varepsilon_{\vec\beta}$, $k_{\vec{\al}}=k_{\vec{\beta}}$,
and $\lgl \lgl p_{\vec{\al}}(i,\delta_{\vec{\al}}(j)):j<k_{\vec{\al}}\rgl:i< d\rgl
=
\lgl \lgl p_{\vec{\beta}}(i,\delta_{\vec{\beta}}(j)):j<k_{\vec{\beta}}\rgl:i< d\rgl$.
\end{proof}

Let $l^*$ denote the length of the nodes $t_{i,j}$.

\begin{claim}\label{claim.j=j'}
Given any $\vec\al,\vec\beta\in \prod_{i<d}K'_i$,
if $j,j'<k^*$ and $\delta_{\vec\al}(j)=\delta_{\vec\beta}(j')$,
 then $j=j'$.
\end{claim}

\begin{proof}
Let $\vec\al,\vec\beta$ be members of $\prod_{i<d}K'_i$   and suppose that
 $\delta_{\vec\al}(j)=\delta_{\vec\beta}(j')$ for some $j,j'<k^*$.
For each $i<d$, let  $\rho_i$ be the relation from among $\{<,=,>\}$ such that
 $\al_i\,\rho_i\,\beta_i$.
Let   $\iota$ be a member of  $\mathcal{I}$  such that for each $\vec\zeta\in[K]^{2d}$ and each $i<d$,
$\zeta_{\iota(2i)}\ \rho_i \ \zeta_{\iota(2i+1)}$.
Take
$\vec\theta\in[K']^{2d}$ satisfying
$\iota_e(\vec\theta)=\vec\al$ and $\iota_o(\vec\theta)= \vec\beta$.
Since between any two members of $K'$ there is a member of $K$, there is a
 $\vec\gamma\in[K]^{d}$ such that  for each $i< d$,
 $\al_i\,\rho_i\,\gamma_i$ and $\gamma_i\,\rho_i\, \beta_i$.
Given that  $\al_i\,\rho_i\,\gamma_i$ and $\gamma_i\,\rho_i\, \beta_i$ for each $i<d$,
there are  $\vec\mu,\vec\nu\in[K]^{2d}$ such that $\iota_e(\vec\mu)=\vec\al$,
$\iota_o(\vec\mu)=\vec\gamma$,
$\iota_e(\vec\nu)=\vec\gamma$, and $\iota_o(\vec\nu)=\vec\beta$.
Since $\delta_{\vec\al}(j)=\delta_{\vec\beta}(j')$,
the pair $\lgl j,j'\rgl$ is in the last sequence in  $f(\iota,\vec\theta)$.
Since $f(\iota,\vec\mu)=f(\iota,\vec\nu)=f(\iota,\vec\theta)$,
also $\lgl j,j'\rgl$ is in the last  sequence in  $f(\iota,\vec\mu)$ and $f(\iota,\vec\nu)$.
It follows that $\delta_{\vec\al}(j)=\delta_{\vec\gamma}(j')$ and $\delta_{\vec\gamma}(j)=\delta_{\vec\beta}(j')$.
Hence, $\delta_{\vec\gamma}(j)=\delta_{\vec\gamma}(j')$,
and therefore $j$ must equal $j'$.
\end{proof}

For any $\vec\al\in \prod_{i<d}K'_i$ and any $\iota\in\mathcal{I}$, there is a $\vec\theta\in[K]^{2d}$ such that $\vec\al=\iota_o(\vec\theta)$.
By homogeneity of $f$ and  by the first sequence in the second line of equation  (\ref{eq.fiotatheta}), there is a strictly increasing sequence
$\lgl j_i:i< d\rgl$  of members of $k^*$ such that for each $\vec\al\in \prod_{i<d}K'_i$,
$\delta_{\vec\al}(j_i)=\al_i$.
For each $i< d$, let $t^*_i$ denote $t_{i,j_i}$.
Then  for each $i<d$ and each $\vec\al\in \prod_{i<d}K'_i$,
\begin{equation}
p_{\vec\al}(i,\al_i)=p_{\vec{\al}}(i, \delta_{\vec\al}(j_i))=t_{i,j_i}=t^*_i.
\end{equation}

\begin{lem}\label{lem.compat}
The set of conditions  $\{p_{\vec{\al}}:\vec{\al}\in \prod_{i<d}K'_i\}$ is  compatible.
\end{lem}

\begin{proof}
Suppose toward a contradiction that there are $\vec\al,\vec\beta\in\prod_{i<d}K'_i$ such that $p_{\vec\al}$ and
 $p_{\vec\beta}$ are incompatible.
By Claim \ref{claim.onetypes},
for each $i<d$ and $j<k^*$,
\begin{equation}
 p_{\vec{\al}}(i,\delta_{\vec{\al}}(j))
=t_{i,j}
=p_{\vec{\beta}}(i,\delta_{\vec{\beta}}(j)).
\end{equation}
Thus,
 the only way $p_{\vec\al}$ and $p_{\vec\beta}$ can be incompatible is if
there are  $i< d$ and $j,j'<k^*$ such that
$\delta_{\vec\al}(j)=\delta_{\vec\beta}(j')$
but
$p_{\vec\al}(i,\delta_{\vec\al}(j))\ne p_{\vec\beta}(i,\delta_{\vec\beta}(j'))$.
Since
$p_{\vec\al}(i,\delta_{\vec\al}(j))=t_{i,j}$ and
$p_{\vec\beta}(i,\delta_{\vec\beta}(j'))= t_{i,j'}$,
this would imply
 $j\ne j'$.
But by Claim \ref{claim.j=j'},
$j\ne j'$ implies that $\delta_{\vec\al}(j)\ne\delta_{\vec\beta}(j')$, a contradiction.
Therefore,
 $p_{\vec\al}$ and $p_{\vec\beta}$ must be  compatible.
\end{proof}

To build the strong subtrees $S_i\sse T_i$,
for each $i<d$,
let $\stem(S_i)=t_i^*$.
Let $l_0$ be the length of the $t_i^*$.
\vskip.1in

\noindent\underline{Induction Assumption}:
Assume  $1\le m<\om$ and
there  are  levels $l_0,\dots, l_{m-1}$
and
 we have constructed   finite strong subtrees $S_i\re l_{m-1}$ of $T_i$, $i<d$,
such that
for each $j<m$,
$c$ takes color $\varepsilon^*$ on each member of
$\prod_{i<d} S_i(l_j)$.
\vskip.1in

For each $i<d$, let $X_i$ denote the set of immediate extensions  in $T_i$ of  the nodes in $S_i(l_{m-1})$.
For each $i<d$,
let $J_i$ be a subset of $K_i'$ with the same size as $X_i$.
For each $i< d$, label the nodes in $X_i$ as
 $\{q(i,\delta):\delta\in J_i\}$.
Let $\vec{J}$ denote  $\prod_{i< d}J_i$.
Notice that for each
$\vec\al\in \vec{J}$ and $i<d$, $q(i,\al_i)\contains t^*_i=p_{\vec{\al}}(i,\al_i)$.

We now construct a condition $q\in\bP$ such that
for each $\vec\al\in\vec{J}$,
$q\le p_{\vec\al}$.
Let
 $\vec{\delta}_q=\bigcup\{\vec{\delta}_{\vec\al}:\vec\al\in \vec{J}\}$.
For each pair $(i,\gamma)$ with $i<d$ and $\gamma\in\vec{\delta}_q\setminus
J_i$,
there is at least one $\vec{\al}\in\vec{J}$ and some $j'<k^*$ such that $\delta_{\vec\al}(j')=\gamma$.
For any other $\vec\beta\in\vec{J}$ for which $\gamma\in\vec{\delta}_{\vec\beta}$,
since the set $\{p_{\vec{\al}}:\vec{\al}\in\vec{J}\}$ is pairwise compatible by Lemma \ref{lem.compat},
it follows
 that $p_{\vec\beta}(i,\gamma)$ must  equal $p_{\vec{\al}}(i,\gamma)$, which is exactly $t^*_{i,j'}$.
%Further note that  in this case,
%$\vec{\delta}_{\vec\beta}(j')$ must also equal  %$\gamma$:
%If  $j''<k^*$ is  the integer  satisfying %$\gamma=\vec{\delta}_{\vec\beta}(j'')$,
%then
%$t^*_{i,j''}=p_{\vec\beta}(i,%\vec{\delta}_{\beta}%(j''))=p_{\vec\beta}(i,%\gamma)=p_{\vec\al}(i,\gamma)=t^*_{i,j'}$, and %hence $j''=j'$.
Let $q(i,\gamma)$ be the leftmost extension
 of $t_{i,j'}^*$ in $T$.
%By the above argument, $q(i,\gamma)$ is well-defined.
%(Note: we could also have required that the conditions in $\bP$ have all nodes corresponding to different $\al$'s be different, and then this second part I deleted would be true.  It's not necessary though.)

Thus, $q(i,\gamma)$ is defined for each pair $(i,\gamma)\in d\times \vec{\delta}_q$.
Define
\begin{equation}
q= \{\lgl (i,\delta),q(i,\delta)\rgl: i<d,\  \delta\in \vec{\delta}_q\}.
\end{equation}

\begin{claim}\label{claim.qbelowpal}
For each $\vec\al\in \vec{J}$,
$q\le p_{\vec\al}$.
\end{claim}

\begin{proof}
Given $\vec\al\in\vec{J}$,
by  our construction
for each pair $(i,\gamma)\in d\times\vec{\delta}_{\vec\al}$, we have
$q(i,\gamma)\contains p_{\vec{\al}}(i,\gamma)$.
\end{proof}

To construct  the $m$-th level  of the strong trees $S_i$,
take an $r\le q$ in  $\bP$ which  decides some $l_m\ge l_q$  for which   $c(\dot{b}_{\vec\al}\re l_m)=\varepsilon^*$, for all $\vec\al\in\vec{J}$.
By extending or truncating $r$, we may assume, without
 loss of generality,  that $l_m$ is equal to the length of the nodes in the image of $r$.
Notice that since
$r$ forces $\dot{b}_{\vec{\al}}\re l_m=\{r(i,\al_i):i<d\}$ for each $\vec\al\in \vec{J}$,
and since the coloring $c$ is defined in the ground model,
it is simply true in the ground model that
$c(\{r(i,\al_i):i<d\})=\varepsilon^*$ for each $\vec\al\in \vec{J}$.
For each $i<d$ and $\al_i\in J_i$,
extend the nodes in $X_i$ to level $l_m$ by extending  $q(i,\delta)$ to $r(i,\delta)$.
Thus, for each $i<d$,
we define $S_i(l_m)=\{r(i,\delta):\delta\in J_i\}$.
It follows that $c$ takes value $\varepsilon^*$ on each member of $\prod_{i<d} S_i(l_m)$.

For each $i<d$,  let $S_i=\bigcup_{m<\om} S_i(l_m)$, and let $L=\{l_m:m<\om\}$.
Then  each $S_i$ is a strong subtree of $T_i$, and
$c$ takes value $\varepsilon^*$ on $\bigcup_{l\in L}\prod_{i<d}S_i(l)$.
\end{proof}

%%%%%%%%%%%%%%%%%%%%%%%%%%%%%%%%%%%%%%%%%%%
%%%%%%%%%%%%%%%%%%%%%%%%%%%%%%%%%%%%%%%%%%%

The Halpern-\Lauchli\ Theorem is used to obtain
a space of strong trees with important Ramsey properties.

\begin{defn}[Milliken space]\label{defn.milliken}
The {\em Milliken space} is the triple
 $(\mathcal{M},\le,r)$, where
$\mathcal{M}$ consists of all infinite strong subtrees $T\sse 2^{<\om}$;
$\le$ is the partial ordering defined by
 $S\le T$ if and only if  $S$ is a subtree of $T$;
and
$r_k(T)$ is
the $k$-th restriction of  $T$, meaning
the set of all
$t\in T$ with  $<k$ splitting nodes in $T$ below $t$.
\end{defn}

Thus, $r_k(T)$ is a finite strong tree with $k$ many levels.
Let $\mathcal{AM}_k$ denote the set of all strong trees with $k$ levels, and let
 $\mathcal{AM}$ denote all finite strong trees.
A topology on $\mathcal{M}$ is generated by basic open sets of the form
\begin{equation}
[U,T]=\{S\in\mathcal{M}:\exists k(r_k(S)=U)\mathrm{\ and\ }S\le T\},
\end{equation}
where $U\in\mathcal{A}$ and $T\in\mathcal{M}$.
Milliken proved in \cite{Milliken81} that, in current terminology, the space of all strong trees  forms a topological Ramsey space.
The  properties of  meager and having the property of Baire are defined in the standard way from the topology.
A subset $\mathcal{X}\sse\mathcal{M}$
is {\em Ramsey} if for every $\emptyset\ne [U,T]$ there is an $S\in[U,T]$ such that either $[U,S]\sse\mathcal{X}$ or else $[U,S]\cap\mathcal{X}=\emptyset$.
$\mathcal{X}$ is {\em Ramsey null} if the second option occurs for any given $[U,S]$.

\begin{defn}[\cite{TodorcevicBK10}]\label{defn.5.2}
A triple $(\mathcal{R},\le,r)$ is a {\em topological Ramsey space} if every subset of $\mathcal{R}$  with the property of Baire  is Ramsey and if every meager subset of $\mathcal{R}$ is Ramsey null.
\end{defn}

The following theorem is a consequence of Milliken's result that the space of all strong subtrees of $2^{<\om}$ is a topological Ramsey space.
This  theorem provides a powerful  tool for obtaining finite big Ramsey degrees
for the Rado graph and  the rationals, considered in the next section.

\begin{thm}[Milliken, \cite{Milliken81}]\label{thm.M}
For each $k<\om$, $T\in\mathcal{M}$, and coloring
of all finite strong subtrees of $T$ with $k$ levels,
there is an infinite strong subtree  $S\le T$ such that
all finite strong subtrees of $S$ with $k$ levels have the same color.
\end{thm}

In the setting of topological Ramsey spaces, the following is the
pigeonhole principle ({\bf Axiom A.4} in \cite{TodorcevicBK10}); it
 follows from an application
of the  Halpern-\Lauchli\ Theorem as shown below.
For $U\in \mathcal{AM}_k$,
$r_{k+1}[U,T]$ denotes the set
$\{r_{k+1}(S):S\in [U,T]\}$.

\begin{lem} \label{lem.A4Milliken}
Let $k<\om$, $U\in\mathcal{AM}_k$, and $T\in\mathcal{M}$ such that
$r_{k+1}[U,T]$ is nonempty.
Then for each coloring of $r_{k+1}[U,T]$ into finitely many colors,
there is an $S\in [U,T]$ such that all members of $r_{k+1}[U,S]$ have the same color.
\end{lem}

\begin{proof}
By induction on the number of colors, it suffices to consider colorings into two colors.
Let $c:r_{k+1}[U,T]\ra 2$ be given.
If $k=0$, then $r_1[U,T]$ is simply the set of nodes in $T$.
In this case, the pigeonhole principle is exactly the Halpern-\Lauchli\ Theorem on one tree.
Now suppose $k\ge 1$.
Note that
there are $2^k$ many immediate successors  of the maximal nodes in $U$;
list
these  as $s_i$, $i<2^k$, and let $T_i=\{t\in T:t\contains s_i\}$.
Let $L$ denote the levels of the trees $T_i$; that is, $L$ is the set of all  the lengths of the nodes in $T_i$ which split.
This set $L$ is the same for each $i<2^k$, since the $T_i$ are cones in the strong tree $T$ starting at the level one above the maximum lengths of nodes in $U$.
Notice that
$r_{k+1}[U,T]$ is exactly the set of all $U\cup \{u_i:i<2^k\}$, where $\lgl u_i:i<2^k\rgl$ is a member of $\prod_{i<2^k}T_i(l)$ for some $l\in L$.
Let $d$  be the coloring on $\bigcup_{l\in L}\prod_{i<2^k}T_i(l)$ induced by $c$ as follows:
\begin{equation}
d(\lgl u_i:i<2^k\rgl)=c(U\cup \{u_i:i<2^k\}).
\end{equation}
Apply the Halpern-\Lauchli\ Theorem for $2^k$ many trees to obtain an infinite set of levels $K\sse L$ and strong subtrees $S_i\sse T_i$ with nodes at the levels in $K$
such that  $d$ is monochromatic on $\bigcup_{l\in K}\prod_{i<2^k}S_i(l)$.
Let
\begin{equation}
S=U\cup\bigcup_{i<2^k}S_i.
\end{equation}
Then $S$ is a strong subtree of $T$ such that $r_k(S)=U$
and $c$ is monochromatic on $r_{k+1}[U,S]$.
\end{proof}

Theorem \ref{thm.M} is obtained by  Lemma \ref{lem.A4Milliken} using induction on $k$.

%%%%%%%%%%%%%%%%%%%%5
%%%%%%%%%%%%%%%%%%%%5
%%%%%%%%%%%%%%%%%%%%5
%%%%%%%%%%%%%%%%%%%%5
%%%%%%%%%%%%%%%%%%%%5
%%%%%%%%%%%%%%%%%%%%5

\section{Applications of Milliken's theorem to homogeneous binary relational structures}\label{sec.Sauer}

The {\em random graph} is the graph on $\om$ many vertices such  that given any two vertices, there is a 50\% chance that there is an edge between them.
A graph $\R$ on $\om$ many vertices is random if and only if it is universal for all
  countable graphs; that is, every countable graph embeds into $\R$.
This is equivalent to $\R$ being
 homogeneous;
any isomorphism between two finite subgraphs of $\R$ can be extended to an automorphism of $\R$.
Another characterization of the random graph is that it
 is the \Fraisse\ limit of the \Fraisse\ class of finite graphs.
As the random graph  on $\om$ many vertices was first constructed by R.\ Rado, we will  call it the {\em Rado graph}, and denote it by $\mathcal{R}$.

The Rado graph has the Ramsey Property for vertex colorings.

\begin{fact}[Folklore]\label{fact.RvertexRP}
For each  coloring of the vertices of the Rado graph $\mathcal{R}$ into finitely many colors,
 there is a subgraph $\mathcal{R}'$ which is  also a Rado graph,
in which the vertices are homogeneous for $c$.
\end{fact}

For finite colorings of the copies of a  graph with more than one vertex, it is not always possible to cut down to one color in a copy of the full Rado graph.
However, Sauer showed that there is a bound on the number of colors that cannot be avoided.

\begin{thm}[Sauer, \cite{Sauer06}]\label{thm.Sauer}
The Rado graph has finite big Ramsey degrees.
\end{thm}

The following outlines the key steps in Sauer's proof:
  Let $\G$ be a finite graph.
\begin{enumerate}
\item
Trees can code graphs.
\item
There are only finitely many  isomorphism types of trees coding $\mathrm{G}$, and only the strongly diagonal types matter.
\item
For each isomorphism type of tree coding $\G$,
there is a way of  enveloping it into a finite strong tree.
\item
The  coloring on copies of  $\G$ can be transferred to  color  finite strong trees, and
  Milliken's Theorem may be applied to
these `strong tree envelopes' of the trees coding $\G$.
\item
Conclude that there is an infinite  strong subtree of $2^{<\om}$ which includes a code of $\mathcal{R}$ and on which there  is one color per isomorphism type of tree coding $\G$.
\end{enumerate}

Let $s$ and $t$ be nodes in $2^{<\om}$, and suppose $|s|<|t|$.
If $s$ and $t$  represent vertices  $v$ and $w$,
then $s$ and $t$ represent an edge between $v$ and $w$ if and only if $t(|s|)=1$.
Thus, if $t(|s|)=0$, then $s$ and $t$ represent no edge between vertices $v$ and $w$.

Let $\G$ be a graph.
Enumerate the vertices of $\G$ in any order as $\lgl v_n:n<N\rgl$, where $N=|\G|$.
Any collection of nodes $\lgl t_n:n<N\rgl$
in $2^{<\om}$ for which the following hold
is a {\em tree coding $\G$}:
For each  pair $m<n<N$,
\begin{enumerate}
\item
$|t_m|<|t_n|$; and
\item
$t_n(|t_m|)=1\Leftrightarrow v_n\ E\ v_m$.
\end{enumerate}
The integer $t_n(|t_m|)$ is called the {\em passing number of $t_n$ at $t_m$}.

%%%%%%%%%%%%%%%%%%%%%%

\begin{figure}\label{figure.3}
\begin{tikzpicture}[grow'=up,scale=.5]
\tikzstyle{level 1}=[sibling distance=4in]
\tikzstyle{level 2}=[sibling distance=2in]
\tikzstyle{level 3}=[sibling distance=1in]
\tikzstyle{level 4}=[sibling distance=0.5in]
\node {$\left< \ \right >$}
 child{
	child{
		child{edge from parent[draw=none]}
		child{ coordinate (t2)
			}
		}
	child{ coordinate (t1)
		child{
			child{edge from parent[draw=none]}
			child{coordinate (t3)}}
		child{edge from parent[draw=none]} }}
 child{coordinate (t0)
	};
		
\node[right] at (t0) {$t_{0}$};
\node[right] at (t1) {$t_{1}$};
\node[right] at (t2) {$t_{2}$};
\node[right] at (t3) {$t_{3}$};

\node[circle, fill=black,inner sep=0pt, minimum size=6pt] at (t0) {};
\node[circle, fill=black,inner sep=0pt, minimum size=6pt] at (t1) {};
\node[circle, fill=black,inner sep=0pt, minimum size=6pt] at (t2) {};
\node[circle, fill=black,inner sep=0pt, minimum size=6pt] at (t3) {};

\draw[thick, dotted] let \p1=(t1) in (-12,\y1) node (v1) {$\bullet$} -- (7,\y1);
\draw[thick, dotted] let \p1=(t2) in (-12,\y1) node (v2) {$\bullet$} -- (7,\y1);
\draw[thick, dotted] let \p1=(t3) in (-12,\y1) node (v3) {$\bullet$} -- (7,\y1);
\draw[thick, dotted] let \p1=(t0) in (-12,\y1) node (v0) {$\bullet$} -- (7,\y1);

\node[left] at (v1) {$v_1$};
\node[left] at (v2) {$v_2$};
\node[left] at (v3) {$v_3$};
\node[left] at (v0) {$v_0$};

\draw[thick] (v0.center) to (v1.center) to (v2.center) to (v3.center) to [bend left] (v0.center);

\end{tikzpicture}
\caption{A tree coding a 4-cycle}
\end{figure}
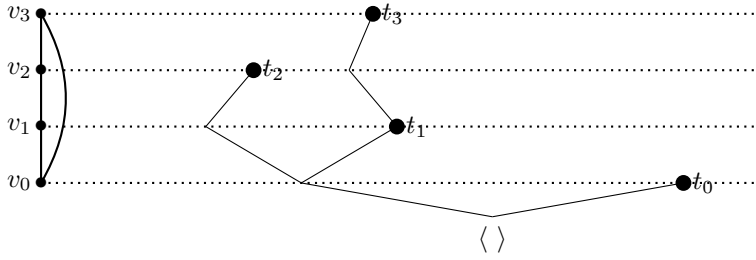

A tree $Z\sse 2^{<\om}$ is {\em strongly diagonal}
if $Z$ is the meet closure of its terminal nodes,
no two terminal nodes of $Z$ have the same length,
for each $l<\om$, there is at most one splitting node or terminal node in $Z$ of length $l$,
and at a splitting node, all other nodes not splitting at that level have passing number $0$.

\begin{defn}[Sauer, \cite{Sauer06}]\label{def.Sauer}
Let $S$ and $T$ be subtrees of $2^{<\om}$.
A function $f:S\ra T$ is a {\em strong similarity} of $S$ to $T$ if for all nodes $s,t,u,v\in S$,
\begin{enumerate}
\item
$f$ is a bijection.
\item
($f$ preserves initial segments)
$s\wedge t\sse u\wedge v$ if and only if $f(s)\wedge f(t)\sse f(u)\wedge f(v)$.
\item
($f$ preserves relative lengths)
$|s\wedge t|<|u\wedge v|$ if and only if
$|f(s)\wedge f(t)|<|f(u)\wedge f(v)|$.
\item
($f$ preserves passing numbers)
If $|u|>|s\wedge t|$,
then $f(u)(|f(s\wedge t)|)=u(|s\wedge t|)$.
\end{enumerate}

\end{defn}
Whenever there is a strong similarity of $S$ to $T$, we call $S$ a {\em copy} of $T$.

Each finite strongly diagonal $X$ tree may be enveloped into a strong tree.
Minimal envelopes will have the same number of levels as the number of meets and maximal nodes in $X$.
The following diagram provides an example of the similarity type of a tree coding an edge, where the
leftmost node is longer than the rightmost.
In this example, an edge is coded
by the nodes $(010)$ and $(0001)$,
since the passing number of the sequence $(0001)$  at length  $3=|(010)|$ is $1$.
There is only one possible envelope for this particular  example of this strong similarity type coding an edge.
\vskip.2in

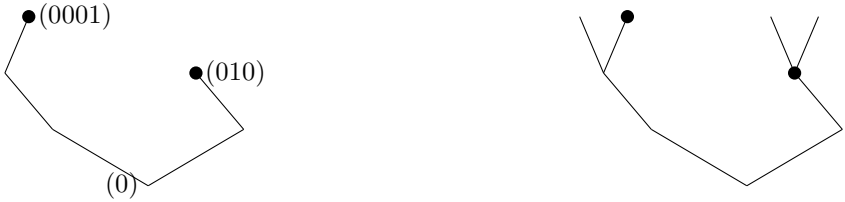
\begin{minipage}[b]{0.5\textwidth}
\centering
\begin{tikzpicture}[grow'=up,scale=.5]
\tikzstyle{level 1}=[sibling distance=4in]
\tikzstyle{level 2}=[sibling distance=2in]
\tikzstyle{level 3}=[sibling distance=1in]
\tikzstyle{level 4}=[sibling distance=0.5in]
\node {}
child{ coordinate (t0)  edge from parent[draw=none]
child{
child{
child{edge from parent[draw=none]}
child{coordinate (t0001)}}
child{ edge from parent[draw=none]
child{edge from parent[draw=none]}
child{edge from parent[draw=none]}}}
child{
child{ coordinate(t010)
child{edge from parent[draw=none]}
child{edge from parent[draw=none]}}
child{ edge from parent[draw=none]
child{edge from parent[draw=none]}
child{edge from parent[draw=none]}}}}
child{edge from parent[draw=none]
child{edge from parent[draw=none]
child{edge from parent[draw=none]
child{edge from parent[draw=none]}
child{edge from parent[draw=none]}}
child{edge from parent[draw=none]
child{edge from parent[draw=none]}
child{edge from parent[draw=none]}}}
child{edge from parent[draw=none]
child{edge from parent[draw=none]
child{edge from parent[draw=none]}
child{edge from parent[draw=none]}}
child{edge from parent[draw=none]
child{edge from parent[draw=none]}
child{edge from parent[draw=none]}}} };
\node[right] at (t0001) {$(0001)$};
\node[left] at (t0) {$(0)$};
\node[right] at (t010) {$(010)$};
\node[circle, fill=black,inner sep=0pt, minimum size=5pt] at (t010) {};
\node[circle, fill=black,inner sep=0pt, minimum size=5pt] at (t0001) {};
\end{tikzpicture}
%\captionsetup{font=footnotesize}
%\captionof{figure}{A diagonal tree $D$ coding an edge between two vertices}
\end{minipage}
\begin{minipage}[b]{0.5\textwidth}
\centering
\begin{tikzpicture}[grow'=up,scale=.5]
\tikzstyle{level 1}=[sibling distance=4in]
\tikzstyle{level 2}=[sibling distance=2in]
\tikzstyle{level 3}=[sibling distance=1in]
\tikzstyle{level 4}=[sibling distance=0.5in]
\node{}
child{coordinate (t0) edge from parent[draw=none]
child{coordinate (t00)
child{coordinate (t000)
child{coordinate (t0000)edge from parent[color=black]}
child{coordinate (t0001)}}
child{edge from parent[draw=none]
child{edge from parent[draw=none]}
child{edge from parent[draw=none]}}}
child{ coordinate(t01)edge from parent[color=black]
child{ coordinate(t010)edge from parent[color=black]
child{coordinate(t0100)edge from parent[color=black]}
child{coordinate(t0101)edge from parent[color=black]}}
child{ edge from parent[draw=none]
child{edge from parent[draw=none]}
child{edge from parent[draw=none]}}}}
child{edge from parent[draw=none]
child{edge from parent[draw=none]
child{edge from parent[draw=none]
child{edge from parent[draw=none]}
child{edge from parent[draw=none]}}
child{edge from parent[draw=none]
child{edge from parent[draw=none]}
child{edge from parent[draw=none]}}}
child{edge from parent[draw=none]
child{edge from parent[draw=none]
child{edge from parent[draw=none]}
child{edge from parent[draw=none]}}
child{edge from parent[draw=none]
child{edge from parent[draw=none]}
child{edge from parent[draw=none]}}} };
\node[circle, fill=black,inner sep=0pt, minimum size=5pt] at (t010) {};
\node[circle, fill=black,inner sep=0pt, minimum size=5pt] at (t0001) {};
\end{tikzpicture}
%\captionsetup{font=footnotesize}
%\captionof{figure}{The strong tree enveloping $D$}
\end{minipage}

For some strongly diagonal trees, there  can be more than one minimal envelope.  The pair of nodes $(0)$ and $(110)$ have induce a tree with the second strong similarity type of a tree coding an edge.
On the right is one envelope.
\vskip.15in

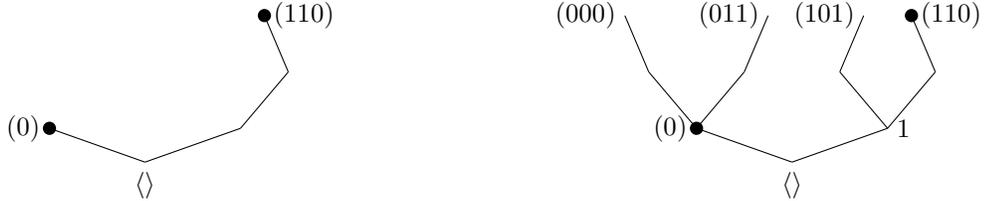
\begin{minipage}[b]{0.5\textwidth}
\centering
\begin{tikzpicture}[grow'=up,scale=.5]
\tikzstyle{level 1}=[sibling distance=2in]
\tikzstyle{level 2}=[sibling distance=1in]
\tikzstyle{level 3}=[sibling distance=.5in]
\node {$\lgl \rgl$}
child{coordinate (t0)
		}
		child{ coordinate(t1)
			child{ edge from parent[draw=none]}
			child{  coordinate(t11)
child{ coordinate(t110) }
child{ edge from parent[draw=none]}} };
\node[left] at (t0) {$(0)$};
\node[right] at (t110) {$(110)$};
\node[circle, fill=black,inner sep=0pt, minimum size=5pt] at (t0) {};
\node[circle, fill=black,inner sep=0pt, minimum size=5pt] at (t110) {};
\end{tikzpicture}
%\captionsetup{font=footnotesize}
%\captionof{figure}{A diagonal tree $D$ coding an edge between two vertices}
\end{minipage}
\begin{minipage}[b]{0.5\textwidth}
\centering
\begin{tikzpicture}[grow'=up,scale=.5]
\tikzstyle{level 1}=[sibling distance=2in]
\tikzstyle{level 2}=[sibling distance=1in]
\tikzstyle{level 3}=[sibling distance=.5in]
\node {$\lgl \rgl$}
child{coordinate (t0)
			child{coordinate (t00)
child{coordinate (t000)}
child{ edge from parent[draw=none]}}
			child{ coordinate(t01)
child{ edge from parent[draw=none]}
child{ coordinate(t011)}}}
		child{ coordinate(t1)
			child{ coordinate(t10)
child{ edge from parent[draw=none]}
child{ coordinate(t101)}}
			child{  coordinate(t11)
child{ coordinate(t110)}
child{  edge from parent[draw=none]}} };
\node[left] at (t0) {$(0)$};
\node[left] at (t000) {$(000)$};
\node[left] at (t011) {$(011)$};
\node[right] at (t1) {$1$};
\node[left] at (t101) {$(101)$};
\node[right] at (t110) {$(110)$};
\node[circle, fill=black,inner sep=0pt, minimum size=5pt] at (t0) {};
\node[circle, fill=black,inner sep=0pt, minimum size=5pt] at (t110) {};
\end{tikzpicture}
%\captionsetup{font=footnotesize}
%\captionof{figure}{The strong tree enveloping $D$}
\end{minipage}
\vskip.1in

Here are two more envelopes for the same tree induced by the nodes $(0)$ and $(110)$ coding an edge.
\vskip.1in

%%%%%%%%%%%%%%%%%%%%%%%%

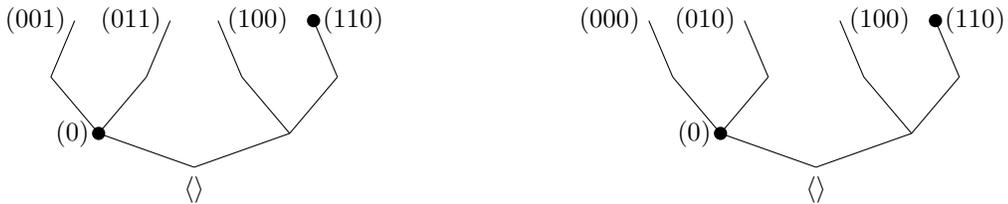
\begin{minipage}[b]{0.5\textwidth}
\centering
\begin{tikzpicture}[grow'=up,scale=.5]
\tikzstyle{level 1}=[sibling distance=2in]
\tikzstyle{level 2}=[sibling distance=1in]
\tikzstyle{level 3}=[sibling distance=.5in]
\node {$\lgl \rgl$}
child{coordinate (t0)
			child{coordinate (t00)
child{ edge from parent[draw=none]}
child{ coordinate(t001)}}
			child{ coordinate(t01)
child{  edge from parent[draw=none]}
child{ coordinate(t011)}}}
		child{ coordinate(t1)
			child{ coordinate(t10)
child{ coordinate(t100)}
child{  edge from parent[draw=none]}}
			child{  coordinate(t11)
child{ coordinate(t110)}
child{   edge from parent[draw=none]}} };
\node[left] at (t0) {$(0)$};
\node[left] at (t001) {$(001)$};
\node[left] at (t011) {$(011)$};
\node[right] at (t100) {$(100)$};
\node[right] at (t110) {$(110)$};
\node[circle, fill=black,inner sep=0pt, minimum size=5pt] at (t0) {};
\node[circle, fill=black,inner sep=0pt, minimum size=5pt] at (t110) {};
\end{tikzpicture}
\end{minipage}
\begin{minipage}[b]{0.5\textwidth}
\centering
\begin{tikzpicture}[grow'=up,scale=.5]
\tikzstyle{level 1}=[sibling distance=2in]
\tikzstyle{level 2}=[sibling distance=1in]
\tikzstyle{level 3}=[sibling distance=.5in]
\node {$\lgl \rgl$}
child{coordinate (t0)
			child{coordinate (t00)
child{coordinate (t000)}
child{   edge from parent[draw=none]}}
			child{ coordinate(t01)
child{ coordinate(t010)}
child{   edge from parent[draw=none]}}}
		child{ coordinate(t1)
			child{ coordinate(t10)
child{ coordinate(t100)}
child{  edge from parent[draw=none]}}
			child{  coordinate(t11)
child{ coordinate(t110)}
child{ edge from parent[draw=none]}} };
\node[left] at (t0) {$(0)$};
\node[left] at (t000) {$(000)$};
\node[left] at (t010) {$(010)$};
\node[right] at (t100) {$(100)$};
\node[right] at (t110) {$(110)$};
\node[circle, fill=black,inner sep=0pt, minimum size=5pt] at (t0) {};
\node[circle, fill=black,inner sep=0pt, minimum size=5pt] at (t110) {};
\end{tikzpicture}
\end{minipage}

%%%%%%%%%%%%%%%%%%%%%%%%%%%%
%%%%%%%%%%%%%%%%%%%%%%%%%%%

The point is that given a strong similarity type of a finite  strongly diagonal tree $D$ coding a graph $\G$,
if $k$ is the number of maximal and splitting nodes in  $D$,
then given any strong tree $U$ isomorphic to $2^{\le k}$, there is exactly one copy of $D$ sitting inside of
$X$.
Thus, a coloring on all strongly similar  copies of $D$
inside a strong  tree $S$ may be transfered
to the collection of all finite strong subtrees of $S$ in
$\mathcal{AM}_k$.
Then Milliken's Theorem may be applied to this coloring on all copies of $2^{\le k}$ inside $S$,
obtaining a strong subtree $S'\le S$ in which each copy of $D$ has the same color.

The final step for Sauer's proof is to show that in any infinite strong subtree of $2^{<\om}$, there is an infinite strongly diagonal tree $\bD$ whose terminal nodes code the Rado graph.
Every finite subtree of
$\bD$ will automatically be strongly diagonal, in particular those coding $\G$.
As there are only finitely many strong similarity types of strongly diagonal trees coding a fixed finite graph $\G$,
this provides the upper bound for the big Ramsey degree $n(\G)$.
It also provides the lower bound as
each strongly diagonal type persists in any subtree of $\bD$ which codes the Rado graph.
\vskip.1in

Sauer's Theorem \ref{thm.Sauer} was recently applied to
obatin the following.

\begin{thm}[Dobrinen, Laflamme, Sauer, \cite{Dobrinen/Laflamme/Sauer16}]\label{thm.DLS}
 The Rado graph has the rainbow Ramsey property.
\end{thm}

This means that given any $2\le k<\om$, a finite graph $\G$, and any coloring of all copies of $\G$ in $\mathcal{R}$ into $\om$ many colors, where each color appears at most $k$ times,
then
there is a copy $\mathcal{R}'$ of $\mathcal{R}$ in which each color on a copy of $\G$ appears at most once.
This result extends, with not much more work, to the larger class of binary relational simple structures.
These include the random directed graph and the random tournament.
It is well-known that the natural numbers have the rainbow Ramsey property, as the proof follows from the Ramsey property.
On the other hand, Theorem \ref{thm.Sauer}
 shows that  the Rado graph does not have the Ramsey property when the subgraph being colored has more than one vertex.
However, the finite Ramsey degrees provides  enough strength to still deduce the rainbow Ramsey property for the Rado graph.

We close this section with  two applications of Milliken's Theorem   to colorings of finite subsets of the rationals.
The rational numbers $\bQ=(Q,\le_Q)$
is up to isomorphism the countable dense linear order without endpoints.
The nodes in the tree $2^{<\om}$ may be given the following linear ordering so as to produce a copy of the rationals.

\begin{defn}\label{def.Devlin}
Let $<_Q$ be the order on $2^{<\om}$ defined as follows.  For $s$ and $t$ incomparable, $s<_Q t$ if and only if $s<_{\mathrm{lex}} t$.
If $s\subset t$,
then $s<_Q t$ if and only if $t(|s|)=1$;
and $t<_Q s$ if and only if $t(|s|)=0$.
\end{defn}

With this ordering, $(2^{<\om},<Q)$ is order isomorphic to the rationals.
Then Milliken's Theorem may be applied to deduce Ramsey properties of the rationals.
We do not go into detail on this, but refer the interested reader to \cite{TodorcevicBK10} for further exposition.
The numbers $t_d$ are the {\em tangent numbers}.

\begin{thm}[Devlin, \cite{DevlinThesis}]\label{thm.DevlinThesis}
Let $d$ be a positive integer.
For each coloring of $[\mathbb{Q}]^d$ into finitely many colors, there is a subset $Q'\sse\mathbb{Q}$ which is order-isomorphic to $\mathbb{Q}$ and such that $[Q']^d$ uses at most $t_d$ colors.
There is a coloring of $[\mathbb{Q}]^d$ with $t_d$ colors, none of which can be avoided by going to an order-isomorphic copy of $\mathbb{Q}$.
\end{thm}

Milliken's Theorem along with the ordering  $(2^{<\om},<_Q)$ isomorphic to the rationals  were used to obtain the big Ramsey degrees for finite products of the rationals.

\begin{thm}[Laver, \cite{Laver84}]\label{thm.Laver}
Let $d$ be a positive integer
For each coloring of $\mathbb{Q}^d$, the product of $d$ many copies of the rationals, into finitely many colors,
there are subsets $Q_i\sse\mathbb{Q}$, $i<d$, also forming sets of rationals, such that
no more than $(d+1)!$ colors occur on $\prod_{i<d}Q_i$.
Moreover, $(d+1)!$ is optimal.
\end{thm}

%%%%%%%%%%%%%%%%%%%%
%%%%%%%%%%%%%%%%%%%%
%%%%%%%%%%%%%%%%%%%%
%%%%%%%%%%%%%%%%%%%%
%%%%%%%%%%%%%%%%%%%%

\section{The universal triangle-free graph and Ramsey theory for strong coding trees}\label{sec.H_3}

The problem of whether or not countable homogeneous structures omitting a certain type of substructure can have finite big Ramsey degrees is largely open.
The simplest homogeneous   relational structure  omitting a type
is the universal triangle-free graph.
A graph $\G$ is {\em triangle-free} if  for any three vertices in $\G$, at least one pair  has no edge between them.
The {\em universal triangle-free graph} is the triangle-free graph on $\om$ many vertices  into which every other  triangle-free graph on countably many vertices embeds.
This is the analogue of the Rado graph for triangle-free graphs.
The first universal triangle-free graph was constructed by Henson in \cite{Henson71}, which
 we denote  as $\mathcal{H}_3$, in which
 he proved that any two
 countable universal  triangle-free graphs are isomorphic.

There are several equivalent characterizations  of $\mathcal{H}_3$.
Let
$\mathcal{K}_3$ denote  the \Fraisse\ class of all finite  triangle-free graphs.
We say that a triangle-free graph
$\mathcal{H}$ is {\em homogeneous}  for $\mathcal{K}_3$ if  any isomorphism between two finite subgraphs of $\mathcal{H}$ can be extended to an automorphism of $\mathcal{H}$.

\begin{thm}[Henson, \cite{Henson71}]\label{thm.Henson}
Let $\mathcal{H}$ be a triangle-free graph on $\om$ many vertices.
The following are equivalent.
\begin{enumerate}
\item
$\mathcal{H}$ is universal for countable triangle-free graphs.
\item
$\mathcal{H}$ is the \Fraisse\ limit of  $\mathcal{K}_3$.
\item
$\mathcal{H}$ is homogeneous for $\mathcal{K}_3$.
\end{enumerate}
\end{thm}

In 1971, Henson proved  in \cite{Henson71}  that for any coloring of the vertices of $\mathcal{H}_3$ into two colors, there is either a subgraph $\mathcal{H}'\le \mathcal{H}_3$ with all vertices in the first color, and which is isomorphic to $\mathcal{H}_3$;
or else there is an infinite subgraph $\mathcal{H}'\le\mathcal{H}_3$ in which all the vertices have the second color and  into which each member of $\mathcal{K}_3$ embeds.
Fifteen years later, Komj\'{a}th and \Rodl\ proved that vertex colorings of  $\mathcal{H}_3$  have the Ramsey property.

\begin{thm}[Komj\'{a}th/\Rodl, \cite{Komjath/Rodl86}]\label{thm.KR}
For any finite coloring of the vertices $|\mathcal{H}_3|$,
there is an $\mathcal{H}\in\binom{\mathcal{H}_3}{\mathcal{H}_3}$ such that $|\mathcal{H}|$ has one color.
\end{thm}

The next question was whether
finite colorings of edges in $\mathcal{H}_3$ could be
 reduced to one color on a copy of $\mathcal{H}_3$.
In 1988, Sauer proved that this was impossible.

\begin{thm}[Sauer, \cite{Sauer98}]\label{thm.SauerH_3}
For any finite coloring of the edges in $\mathcal{H}_3$, there is an
 $\mathcal{H}\in\binom{\mathcal{H}_3}{\mathcal{H}_3}$ such that the edges in $\mathcal{H}$ take on no more than two colors.
Furthermore, there is a coloring on the edges in $\mathcal{H}_3$ into two colors such that every universal triangle-free subgraph of $\mathcal{H}_3$ has edges of both colors.
\end{thm}

This is in contrast to a theorem of \Nesetril\ and \Rodl\ in \cite{Nesetril/Rodl77} and \cite{Nesetril/Rodl83} proving that
 the \Fraisse\ class of finite ordered triangle-free graphs has the Ramsey property.
Sauer's result begged  the question of whether
his result would extend to all finite triangle-free graphs.
In other words, does
 $\mathcal{H}_3$ have finite big Ramsey degrees?
This was recently solved by the author in \cite{DobrinenH_317}.

\begin{thm}[Dobrinen, \cite{DobrinenH_317}]\label{thm.DG}
The universal triangle-free graph has finite big Ramsey degrees.
\end{thm}

The proof of Theorem \ref{thm.DG}
proceeds via the following steps.
\begin{enumerate}
\item
Build a space of new kinds of trees, each of which codes $\mathcal{H}_3$.  We call these {\em strong coding trees}.
Develop a new notion of strict similarity type, which  augments the notion of strong similarity type in $2^{<\om}$.
\item
Prove analogues of Halpern-\Lauchli\ and Milliken for
the collection of strong coding trees, obtaining one color per strong similarity type.
These use the technique of forcing. The set-up and arguments are similar to those  in Theorem \ref{thm.matrixHL}.
The coding nodes present an obstacle which must be overcome in several separate forcings; furthermore,
 the partial orderings are stricter  than simply extension.
\item
Develop new notion of envelope.
\item
Apply theorems and notions from (3) and (4) to obtain a strong coding tree $S$ with one color per strict similarity type.
\item
Construct a diagonal subtree  $D\sse S$ which codes $\mathcal{H}_3$, and has room for the envelopes to fit in an intermediary subtree $S'$, where $D\sse S'\sse S$.
\item
Conclude that $\mathcal{H}_3$ has finite big Ramsey degrees.
\end{enumerate}

In this article, we  present the space of strong coding trees and the Halpern-\Lauchli\ analogue, leaving the reader interested in the further steps to read \cite{DobrinenH_317}.

One constraint for finding the finite big Ramsey degrees for $\mathcal{H}_3$ was  that, unlike the bi-embeddability between the Rado graph and  the graph coded by all the nodes in $2^{<\om}$, an interplay which was of fundamental  importance  to Sauer's proof in the previous section,
there is no graph induced by a homogeneous tree of some simple form which is bi-embeddable with the universal triangle-free graph.
Thus, our approach was to consider certain nodes in our trees as distinguished to code vertices, calling them
{\em coding nodes}.
We now present the new space of strong triangle-free trees coding $\mathcal{H}_3$.

For $i<j<k$,  suppose the vertices
 $\{v_i,v_j,v_k\}$  are coded by
the distinguished nodes $t_i,t_j,t_k$ in $2^{<\om}$,
where $|t_i|<|t_j|<|t_k|$.
 The vertices
 $\{v_i,v_j,v_k\}$  form a triangle
if and only if
there are edges between each pair of vertices
if and only if
the distinguished coding nodes $t_i,t_j,t_k$ satisfy
\begin{equation}\label{eq.triangle}
t_k(|t_j|)=t_k(|t_i|)=t_j(|t_i|)=1.
\end{equation}
Whenever
$t_k(|t_i|)=t_j(|t_i|)=1$, we say that $t_k$ and $t_j$ have {\em parallel 1's}.
The following criterion guarantees that as we construct a tree with distinguished nodes coding vertices, we can construct one in which the coding nodes  code no triangles.
\vskip.1in

\noindent \bf Triangle-Free Extension Criterion: \rm  A node $t$ at the level of the $n$-th coding  node $t_n$ extends right if and only if $t$ and $t_n$ have no parallel $1$'s.
\vskip.1in

The following  slight modification  of a property which Henson proved guarantees a copy of $\mathcal{H}_3$ is used in our construction of strong coding trees.
\vskip.1in

\begin{enumerate}
\item[]
\begin{enumerate}
\item[$(A_3)^{\tt{tree}}$]
Let  $\lgl F_i:i<\om\rgl$   be any listing of finite subsets of $\om$ such that $F_i\sse i$ and each finite set  appears as $F_i$ for infinitely many indices $i$.
For each $i<\om$,
if
 $t_k(l_j)=0$ for all pairs $j<k$ in $F_i$,
then
there is some $n\ge i$ such that
for all $k<i$,
$t_n(l_k)=1$ if and only if $ k\in F_i$.
\end{enumerate}
\end{enumerate}
\vskip.1in

We now show how to build a tree $\bS$ which has distinguished nodes coding the vertices in $\mathcal{H}_3$ and which is maximally branching subject to not coding any triangles.
Moreover, the coding nodes will be dense in the tree $\bS$.
Let $\lgl F_i:i<\om\rgl$ list $[\om]^{<\om}$ in such a way that $F_i\sse i$ and each finite set appears infinitely many times.
Enumerate the nodes in $2^{<\om}$ as $\lgl u_i:i<\om\rgl$,
where all nodes in $2^k$ appear before any node in $2^{k+1}$.
Let the first two levels be $2^{\le 1}$ and let the least coding node $t_0$ be $\lgl 1\rgl$.
Extend $\lgl 0\rgl$ both right and left and extend $t_0$ only left.
$t_0$ codes the vertex $v_0$.
From here,
on odd steps $2n+1$,
if the node $u_n$ is in the part of the tree constructed so far,
extend $u_n$ to a coding node $t_{2n+1}$ in $2^{2n+2}$ such that the only edge  it codes is an edge with vertex $v_{2n}$;
that is, $t_{2n+1}(|t_i|)=1$ if and only if $i=2n$.
On even steps $2n$, if $\{t_i:i\in F_n\}$ does not code any edges between the vertices $\{v_i:i\in F_n\}$,
then take some node $s$ in the tree constructed so far such that $s(|t_i|)=1$ for all $i\in F_n$,
$s(|t_i|)=0$ for all $i\in 2n-1$,
and $s(2n)=1$.
That such a node $s$ is in the tree constructed so far is guaranteed by our maximal branching subject to the Triangle-Free Extension Criterion.
If on either step the condition is not met,  then let $t_{2n+1}={0^{2n+1}}^{\frown}1$.

The trees constructed in \cite{DobrinenH_317} have an additional requirement, but these are the main ideas.
Since the condition $(A_3)^{\tt tree}$ is met, the coding nodes in $\bS$ code $\mathcal{H}_3$.
Notice that any subtree of $\bS$ which is isomorphic to $\bS$,  coding nodes being taken into account in the isomorphism, also codes $\mathcal{H}_3$.

\begin{center}
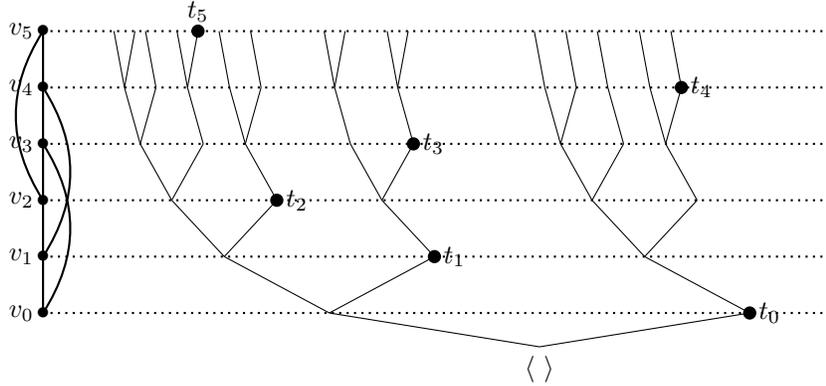
\begin{figure}
\begin{tikzpicture}[grow'=up,xscale=.55,yscale=.5]
\tikzstyle{level 1}=[sibling distance=4in]
\tikzstyle{level 2}=[sibling distance=2in]
\tikzstyle{level 3}=[sibling distance=1in]
\tikzstyle{level 4}=[sibling distance=0.6in]
\tikzstyle{level 5}=[sibling distance=0.3in]
\tikzstyle{level 6}=[sibling distance=0.2in]
\node {$\left< \ \right >$}
 child{
	child{
		child{
			child{
				child{child child}
				child{child child{edge from parent[draw=none]}}}
			child{
				child{child child{coordinate (t5)}}
				child {edge from parent[draw=none]}
			}
			}
		child{ coordinate (t2)
			child{
				child{ child child{edge from parent[draw=none]} }
				child{ child child{edge from parent[draw=none]} }
				}
			child {edge from parent[draw=none]}}
		}
	child{ coordinate (t1)
		child{
			child{
				child{
					child
					child
					}
				child {edge from parent[draw=none]}}
			child{coordinate (t3)
				child{
					child					
					child}
				child {edge from parent[draw=none]}}}
		child {edge from parent[draw=none]} }}
 child{coordinate (t0)
	child{
		child{
			child{
				child{
					child
					child {edge from parent[draw=none]}}
				child{
					child
					child {edge from parent[draw=none]}}}
			child{child{
					child
					child {edge from parent[draw=none]}} child {edge from parent[draw=none]}}
		}
		child{
			child{child{
					child
					child {edge from parent[draw=none]}} child{ coordinate (t4)
					child
					child {edge from parent[draw=none]}}}
			child {edge from parent[draw=none]}}
	}
	child {edge from parent[draw=none]}};
		
\node[right] at (t0) {$t_{0}$};
\node[right] at (t1) {$t_{1}$};
\node[right] at (t2) {$t_{2}$};
\node[right] at (t3) {$t_{3}$};
\node[right] at (t4) {$t_{4}$};
\node[above] at (t5) {$t_{5}$};

\node[circle, fill=black,inner sep=0pt, minimum size=5pt] at (t0) {};
\node[circle, fill=black,inner sep=0pt, minimum size=5pt] at (t1) {};
\node[circle, fill=black,inner sep=0pt, minimum size=5pt] at (t2) {};
\node[circle, fill=black,inner sep=0pt, minimum size=5pt] at (t3) {};
\node[circle, fill=black,inner sep=0pt, minimum size=5pt] at (t4) {};
\node[circle, fill=black,inner sep=0pt, minimum size=5pt] at (t5) {};

\draw[thick, dotted] let \p1=(t1) in (-12,\y1) node (v1) {$\bullet$} -- (7,\y1);
\draw[thick, dotted] let \p1=(t2) in (-12,\y1) node (v2) {$\bullet$} -- (7,\y1);
\draw[thick, dotted] let \p1=(t3) in (-12,\y1) node (v3) {$\bullet$} -- (7,\y1);
\draw[thick, dotted] let \p1=(t0) in (-12,\y1) node (v0) {$\bullet$} -- (7,\y1);
\draw[thick, dotted] let \p1=(t4) in (-12,\y1) node (v4) {$\bullet$} -- (7,\y1);
\draw[thick, dotted] let \p1=(t5) in (-12,\y1) node (v5) {$\bullet$} -- (7,\y1);

\node[left] at (v1) {$v_1$};
\node[left] at (v2) {$v_2$};
\node[left] at (v3) {$v_3$};
\node[left] at (v0) {$v_0$};
\node[left] at (v4) {$v_4$};
\node[left] at (v5) {$v_5$};

\draw[thick] (v0.center) to (v1.center) to (v2.center) to (v3.center) to [bend left] (v0.center);
\draw[thick] (v3.center) to (v4.center) to (v5.center) to [bend right] (v2.center);

\draw[thick] (v4.center) to[bend left] (v1.center);

\end{tikzpicture}
\caption{Initial part of  a tree $\bS$ coding $\mathcal{H}_3$ maximally branching subject to TFC}
\end{figure}
\end{center}

The trees of the form $\bS$ are  almost what is needed.
However, in order to procure the full extent of the Ramsey theorems needed,
it is necessary to work with stretched versions of $\bS$ which are skew.
This means that each level has at most one of either a node which splits or a coding node.
Call the set of nodes which are either coding or splitting the set of {\em critical nodes}.
Let $\bT$ denote a skewed version of $\bS$, so that its coding nodes are dense in $\bT$ and code $\mathcal{H}_3$.
Let $\mathcal{T}(\bT)$ denote the collection of all subtrees of $\bT$ which are isomorphic to $\bT$.
Thus, every tree in $\mathcal{T}(\bT)$  codes $\mathcal{H}_3$.

Similarly to the notation for the Milliken space in Definition \ref{defn.milliken} and following,
let $r_k(T)$ denote the first $k$ levels of $T$;
thus $r_k(T)$ contains a total of $k$ critical  nodes.
$r_{k+1}[r_k(T),T]$ denotes the set of all $r_{k+1}(S)$, where $r_k(S)=r_k(T)$ and $S$ is a subtree of $T$ in $\mathcal{T}(\bT)$.
The following is the analogue of the Halpern-\Lauchli\ Theorem for strong coding trees.

\begin{thm}[Dobrinen, \cite{DobrinenH_317}]\label{thm.A.4}
Let $\bT$ be a strong coding tree, $T\in\mathcal{T}(\bT)$,  $k<\om$, and $c$ be a coloring of $r_{k+1}[r_k(T),T]$ into two colors.
Then there is an $S\in [r_k(T),T]$ such that all members of $r_{k+1}[r_k(T),S]$ have the same $c$-color.
\end{thm}

The case when the maximal critical node in $r_{k+1}(T)$ is a splitting node is significantly simpler to handle than when it is a coding node, so we present that forcing here.
Enumerate the maximal nodes of $r_{k+1}(T)$ as $\lgl s_i:i\le d\rgl$, and let $s_d$ denote the splitting node.
Let $i_0$ denote the index such that $s_{i_0}\in \{0\}^{<\om}$.
For each $i\ne i_0$, let $T_i=\{t\in T:t\contains s_i\}$;
let $T_{i_0}={0}^{<\om}$.
Let $L$ denote the set of levels $l$ such that there is a splitting node at level $l$ in $T_d$.

Let $\bP$ be the set of conditions $p$ such that
$p$ is a  function
of the form
$$
p:\{d\}\cup(d\times\vec{\delta}_p)\ra T\re l_p,
$$
where $\vec{\delta}_p\in[\kappa]^{<\om}$ and $l_p\in L$,
such that
\begin{enumerate}
\item[(i)]
$p(d)$ is {\em the} splitting  node extending $s_d$ at level $l_p$;
\item [(ii)]
For each $i<d$,
 $\{p(i,\delta) : \delta\in  \vec{\delta}_p\}\sse  T_i\re l_p$.
\end{enumerate}
\vskip.1in

The partial
 ordering on $\bP$  is defined as follows:
$q\le p$ if and only if either
\begin{enumerate}
\item
$l_q=l_p$ and $q\contains p$ (so also $\vec{\delta}_q\contains\vec{\delta}_p$);
or else
\item
$l_q>l_p$, $\vec{\delta}_q\contains \vec{\delta}_p$, and
\begin{enumerate}
\item[(i)]
$q(d)\supset p(d)$,
and for each $\delta\in\vec{\delta}_p$ and $i<d$,
$q(i,\delta)\supset p(i,\delta)$;
\item[(ii)]
Whenever  $(\al_0,\dots,\al_{d-1})$ is a strictly increasing sequence in $(\vec{\delta}_p)^d$ and
\begin{equation}
 r_k(T)\cup\{p(i,\al_i):i<d\}\cup \{p(d)\}\in
r_{k+1}[r_k(T),T],
\end{equation}
then also
 \begin{equation}
r_k(T)\cup \{q(i,\al_i):i<d\}\cup \{q(d)\}\in
r_{k+1}[r_k(T),T].
\end{equation}
\end{enumerate}
\end{enumerate}

The proof proceeds in a similar manner to that of Theorem \ref{thm.matrixHL}, except that the name for an ultrafilter $\dot{\mathcal{U}}$ is now on $L$ and
there is much checking that certain criteria are met so that the tree being chosen by the forcing will actually again be a member of $\mathcal{T}(\bT)$.
The actual theorem needed to obtain the finite big Ramsey degrees for $\mathcal{H}_3$  is more involved; the statement above appears as a remark after Theorem 22 in \cite{DobrinenH_317}.
It is presented here in the hope that the reader will see the similarity with Harrington's forcing proof.

%%%%%%%%%%%%%%%%%%%%
%%%%%%%%%%%%%%%%%%%%
%%%%%%%%%%%%%%%%%%%%
%%%%%%%%%%%%%%%%%%%%
%%%%%%%%%%%%%%%%%%%%
%%%%%%%%%%%%%%%%%%%%

\section{Halpern-\Lauchli\ Theorem  on trees of uncountable height and applications}\label{sec.measurable}

The Halpern-\Lauchli\ Theorem can be extended to trees of certain uncountable heights.
Here, we present some of the known results and their applications to uncountable homogeneous relational structures, as well as some open problems in this area.

The first extension of the Halpern-\Lauchli\ Theorem to uncountable height trees is due to Shelah.
In fact,
he proved  a
strengthening of Milliken's Theorem, where  level sets of size $m$, for a fixed positive integer $m$, are the objects being colored;  this includes strong trees with $k$ levels when $m=2^k$, as strong trees can be recovered from the meet closures of their maximal nodes.
Shelah's theorem was strengthened to all $m$-sized antichains in \cite{Dzamonja/Larson/MitchellQ09}.
This is the version  presented here.
The definition of {\em strong embedding} in the setting of uncountable height trees is an augmented analogue of Definition \ref{def.Sauer}
 for countable height trees,
with the additional requirement that
a fixed well-ordering $\prec$ on the nodes in $2^{<\kappa}$
is also preserved.
A {\em strong subtree} of $2^{<\kappa}$ is a subtree $S\sse 2^{<\kappa}$ such that there is
a set of levels $L\sse\kappa$ of cardinality $\kappa$
and  the splitting nodes in $S$ are exactly those with length in $L$.
In particular, if $e:2^{<\kappa}\ra 2^{<\kappa}$ is a strong embedding, then
the image $e[2^{<\kappa}]$ is a strong subtree of $2^{<\kappa}$.
For a node $w$ in a strong tree, Cone$(w)$ denotes the set of all nodes in the tree extending $w$.
The  proof of the following theorem
is an elaborate forcing proof,
having as its base
  ideas from  Harrington's forcing  proof of Theorem \ref{thm.matrixHL}.

\begin{thm}[Shelah, \cite{Shelah91}; \Dzamonja, Larson, and Mitchell, \cite{Dzamonja/Larson/MitchellQ09}]\label{thm.Shelah}
Suppose that $m<\om$ and $\kappa$ is a cardinal which is measurable in the generic extension obtained by adding $\lambda$ many Cohen subsets of $\kappa$,
where $\lambda\ra(\kappa)^{2m}_{2^{\kappa}}$.
Then for any coloring $d$ of the $m$-element
antichains of $2^{<\kappa}$ into $\sigma<\kappa$ colors, and for any well-ordering $\prec$ of the levels of $2^{<\kappa}$,
there is a strong embedding $e:2^{<\kappa}\ra 2^{<\kappa}$ and a dense set of elements $w$ such that
\begin{enumerate}
\item
$e(s)\prec e(t)$ for all $s\prec t$ from Cone$(w)$, and
\item
$d(e[A])=d(e[B])$ for all $\prec$-similar $m$-element antichains $A$ and $B$ of Cone$(w)$.
\end{enumerate}
\end{thm}

A model of ZFC in which the hypothesis of the theorem holds may be obtained by starting in a model of ZFC + GCH + $\exists$ a $(\kappa+2m+2)$-strong cardinal, as was shown by Woodin (unpublished).
A cardinal $\kappa$ is {\em $(\kappa +d)$-strong} if there is an elementary embedding $j:V\ra M$ with critical point $\kappa$ such that $V_{\kappa+d}=M_{\kappa+d}$.

Theorem \ref{thm.Shelah}  is applied  in two papers of D\v{z}amonja, Larson, and Mitchell to prove that
the $\kappa$-rationals and the $\kappa$-Rado graph have finite big Ramsey degrees.
The $\kappa$-rationals, $\mathbb{Q}_{\kappa}=(Q,\le_Q)$, is the strongly $\kappa$-dense linear order of size $\kappa$.
The  nodes in the tree $2^{<\kappa}$ ordered by $<_Q$ produces the $\kappa$-rationals,
where $<_Q$ is the same  ordering  as in Definition \ref{def.Devlin}  applied to the tree $2^{<\kappa}$.

\begin{thm}[\Dzamonja, Larson, and Mitchell,
 \cite{Dzamonja/Larson/MitchellQ09}]
In any model of ZFC in which $\kappa$ is measurable after adding $\beth_{\kappa+\om}$ many Cohen subsets to $\kappa$,
for any fixed  positive integer $m$, given any coloring of
of $[\mathbb{Q}_{\kappa}]^m$ into less than $\kappa$ colors, there is a subset $Q^*\sse Q$ such that
$\mathbb{Q}^*=(Q^*,\le_Q)$ is also a strongly $\kappa$ dense linear order, and
such that the members of $[\mathbb{Q}^*]^m$ take only finitely many colors.
Moreover, for each strong similarity type, all members in  $[\mathbb{Q}^*]^m$  with that similarity type have the same color.
\end{thm}

As each similarity type persists in every smaller copy of the $\kappa$-rationals,
 the strong similarity types provide  the exact finite big  Ramsey degree for
colorings of $[\mathbb{Q}_{\kappa}]^m$.

The {\em $\kappa$-Rado graph}, $\mathcal{R}_{\kappa}$,  is the random graph on $\kappa$ many vertices.
Similarly to  the coding of the Rado graph using nodes in the tree $2^{<\om}$ in Section \ref{sec.Sauer},
the $\kappa$-Rado graph can be coded using nodes in $2^{<\kappa}$.
Analogously to Sauer's use of Milliken's Theorem to prove that the Rado graph has finite big Ramsey degrees,
 Theorem   \ref{thm.Shelah}  was an integral part in the proof obtaining finite bounds for colorings of all copies of a given finite graph inside $\mathcal{R}_{\kappa}$.

\begin{thm}[\Dzamonja, Larson, and Mitchell,
 \cite{Dzamonja/Larson/MitchellRado09}]
 In any model of ZFC with a cardinal $\kappa$ which is measurable after adding $\beth_{\kappa+\om}$
many Cohen subsets of $\kappa$,
for any finite graph $\G$
there is a finite number $r^+_{\G}$ such that
for any coloring of the copies of $\G$ in $\mathcal{R}_{\kappa}$ into less than $\kappa$ many colors,
there is a subgraph $\mathcal{R}'_{\kappa}$ which is also a $\kappa$-Rado graph
in which the copies of $\G$ take on at most $r^+_{\G}$ many colors.
\end{thm}

The number $r^+_{\G}$ is the number of strong similarity types of subtrees of $2^{<\kappa}$ which code a copy of $\G$; recall that  in the uncountable context, the fixed linear ordering  on the nodes in $2^{<\kappa}$ is a part of the description of strong similarity type.
\Dzamonja, Larson, and Mitchell  showed that for any graph $\G$ with more than two vertices,  this number $r^+_{\G}$ is strictly greater than the
Ramsey degree for the same graph $\G$ inside the Rado graph.
They conclude \cite{Dzamonja/Larson/MitchellRado09} with the following question.

\begin{question}[\cite{Dzamonja/Larson/MitchellRado09}]
What is the large cardinal strength of the conclusion of Theorem \ref{thm.Shelah}?
\end{question}

The Halpern-\Lauchli\ Theorem for more than one tree can also be extended to uncountable height trees.
Hathaway and the author considered in  \cite{Dobrinen/Hathaway16}
direct analogues
 of the Halpern-\Lauchli\ Theorem to more than one  tree of uncountable height, also considering trees with infinite branching.
A tree   $T \subseteq {^{<\kappa}\kappa}$
is a {\em $\kappa$-tree} if $T$ has cardinality $\kappa$ and every level of $T$ has cardinality less than $\kappa$.
We shall say that a tree
 $T \subseteq {^{<\kappa}\kappa}$ is
 {\em regular} if it is a perfect
 $\kappa$-tree in which every maximal branch has
 cofinality $\kappa$.
For  $\zeta < \kappa$,
let
 $T(\zeta) = T \cap {^\zeta \kappa}$.

Given a regular tree $T \subseteq {^{<\kappa} \kappa}$,
 tree $S \subseteq T$
 is called a {\em strong subtree} of $T$ if $S$ is regular and there is some set of levels
 $L \subseteq \kappa$ cofinal in $\kappa$
such that for each $s\in S$,
\begin{enumerate}
\item
 $s$ splits if and only if $t$ has length $\zeta\in L$, and
\item
for each $\zeta\in L$ and $s\in S(\zeta)$,
$s$ is  maximally branching in $T$.
\end{enumerate}

\begin{defn}
Let $\delta, \sigma > 0$ be ordinals
 and $\kappa$  be an infinite cardinal.
 $\textrm{HL}(\delta,\sigma,\kappa)$
 is the following statement:
Given any sequence
 $\langle T_i \subseteq {^{<\kappa}\kappa} :
 i < \delta \rangle$
 of regular trees and a coloring
 $$
c:\bigcup_{\zeta<\kappa}\prod_{i<\delta} T_i(\zeta)\ra\sigma,
$$
 there exists a sequence of trees
 $\langle S_i : i < \delta \rangle$ and $L\in[\kappa]^{\kappa}$
 such that
\begin{enumerate}
\item
 each $S_i$ is a strong subtree of $T_i$
 as witnessed by  $L\subseteq \kappa$,
 and
\item
 there is some $\sigma' < \sigma$
  such that
 $c$ has color $\sigma'$ on  $\bigcup_{\zeta\in L}\prod_{i<\delta} S_i(\zeta)$.
\end{enumerate}
\end{defn}

\begin{thm}[Dobrinen and Hathaway, \cite{Dobrinen/Hathaway16}]\label{thm.DH}
Let $d\ge 1$ be any finite integer and suppose that $\kappa$ is a $(\kappa+d)$-strong cardinal in a model $V$ of ZFC satisfying GCH.
Then there is a forcing extension in which $\kappa$ remains measurable and HL$(d,\sigma,\kappa)$ holds, for all $\sigma<\kappa$.
\end{thm}

A direct lifting of the proof in
Theorem \ref{thm.matrixHL} would yield Theorem \ref{thm.DH}, but at the expense of assuming a $(\kappa +2d)$-strong cardinal.
In order to bring the large cardinal strength  down
to a $(\kappa+d)$-strong cardinal,
Hathaway and the author combined the method of proof in  Theorem \ref{thm.matrixHL}
 with ideas from the proof  of the Halpern-\Lauchli\ Theorem in \cite{Farah/TodorcevicBK}, using  the measurability  of  $\kappa$ where the methods in
\cite{Farah/TodorcevicBK}
would not directly lift.
It is known that a $(\kappa+d)$-strong
cardinal is necessary to obtain a model of ZFC in which $\kappa$ is measurable after adding $\kappa^{+d}$ many new Cohen subsets of $\kappa$.
Thus, it is intriguing as to whether or not this is the actual consistency strength of the  Halpern-\Lauchli\ Theorem for $d$ trees at a measurable cardinal, or if there is some other means for obtaining HL$(d,\sigma,\kappa)$ at a measurable cardinal $\kappa$.

\begin{problem}[\cite{Dobrinen/Hathaway16}]\label{problem.constrength}
Find the exact consistency strength of HL$(d,\sigma,\kappa)$ for $\kappa$ a measurable cardinal, $d$ a positive integer, and $\sigma<\kappa$.
\end{problem}

We remark that a model of ZFC  with a strongly inaccessible cardinal $\kappa$ where HL$(d,2,\kappa)$
 fails for some $d\ge 1$  is not known at this time.
We also point out that the use of $\delta$ in the definition of  HL$(\delta,\sigma,\kappa)$, rather than just $d$,
 is in reference to the fact that a
theorem
similar to Theorem \ref{thm.DH}  is proved in \cite{Dobrinen/Hathaway16} for infinitely many trees, a result which is of interest in choiceless models of ZF.
It is open whether there is a model of ZF in which the measurable cardinal remains measurable after adding the amount of Cohen subsets of $\kappa$.
See \cite{Dobrinen/Hathaway16}  for this and other open problems.

Various weaker forms of the Halpern-\Lauchli\ Theorem were also investigated in \cite{Dobrinen/Hathaway16}.
The  {\em somewhere dense Halpern-\Lauchli\ Theorem},  SDHL$(\delta,\sigma,\kappa)$,
 is the version where it is only  required to find  levels $l<l'<\kappa$,  nodes $t_i\in T_i(l)$,
and sets $S_i\sse T_i(l')$ such that each immediate successor of $t_i$ in $T_i$ is extended by a unique member of $S_i$,
and all sequences in $\prod_{i<\delta} S_i$ have the same color.

\begin{thm}[Dobrinen and Hathaway, \cite{Dobrinen/Hathaway16}]
For a weakly compact cardinal $\kappa$ and ordinals $0<\delta,\sigma<\kappa$,
SDHL$(\delta,\sigma,\kappa)$ holds if and only if
HL$(\delta,\sigma,\kappa)$.
\end{thm}

Very recently, Zhang extended Laver's Theorem \ref{thm.Laver}
to the uncountable setting.
First, he proved a strengthened version of Theorem \ref{thm.DH}, though using a stronger large cardinal hypothesis.
Then he applied it to prove the following.

\begin{thm}[Zhang, \cite{Zhang17}]\label{thm.Zhang}
Suppose that $\kappa$ is a cardinal which is measurable
 after forcing to add $\lambda$ many Cohen subsets of $\kappa$, where
$\lambda$ satisfies the partition relation
$\lambda\ra(\kappa)^{2d}_{2^{\kappa}}$.
Then for any coloring of the product of $d$ many copies of $\mathbb{Q}_{\kappa}$ into less than $\kappa$ many colors,
there are copies $Q_i$, $i<d$, of $\mathbb{Q}_{\kappa}$ such that the coloring takes on at most $(d+1)!$ colors on $\prod_{i<d}Q_i$.
Moreover, $(d+1)!$ is optimal.
\end{thm}

As discussed previously, it is consistent with a $(\kappa +2d)$-strong cardinal to have a model where $\kappa$ satisfies the hypothesis of Theorem
\ref{thm.Zhang}.
Zhang asked whether the conclusion of this theorem is a consequence of some large cardinal hypothesis.
The author asks the following possibly easier question.

\begin{question}
Is a $(\kappa+d)$-strong cardinal sufficient to produce a model in which
 the conclusion of Theorem \ref{thm.Zhang}
holds?
\end{question}

In  \cite{Dobrinen/Hathaway16}, it was  proved in ZFC that HL$(1,k,\kappa)$ holds for each positive integer $k$ and each weakly compact cardinal $\kappa$.
This was recently extended by Zhang in \cite{Zhang17}
to colorings into less than $\kappa$ many colors; moreover, he proved this for a stronger asymmetric version.
Furthermore, Zhang showed that relative to the existence of a measurable cardinal, it is possible to have HL$(1,\delta,\kappa)$ hold where $\kappa$ is a strongly inaccessible cardinal which is not weakly compact.
Still, the following basic question remains open.

\begin{question}
Is it true in ZFC that if $\kappa$ is strongly inaccessible, then HL$(1,\delta,\kappa)$ holds, for some (or all) $\delta$ with $2\le\delta<\kappa$?
\end{question}

Mapping out all the implications between the various forms of the  Halpern-\Lauchli\ Theorem
at uncountable cardinals, their applications to homogeneous relational structures, and  their relative consistency strengths
is an area ripe for further  exploration.

\bibliographystyle{amsplain}
\bibliography{references}

\providecommand{\bysame}{\leavevmode\hbox to3em{\hrulefill}\thinspace}
\providecommand{\MR}{\relax\ifhmode\unskip\space\fi MR }
\providecommand{\MRhref}[2]{%
  \href{http://www.ams.org/mathscinet-getitem?mr=#1}{#2}
}
\providecommand{\href}[2]{#2}
\begin{thebibliography}{10}

\bibitem{DevlinThesis}
Dennis Devlin, \emph{Some partition theorems for ultrafilters on $\omega$},
  Ph.D. thesis, Dartmouth College, 1979.

\bibitem{DobrinenCreaturetRs15}
Natasha Dobrinen, \emph{Creature forcing and topological {R}amsey spaces},
  Topology and Its Applications \textbf{213} (2016), 110--126, Special issue in
  honor of Alan Dow's 60th birthday.

\bibitem{DobrinenJSL15}
\bysame, \emph{High dimensional {E}llentuck spaces and initial chains in the
  {T}ukey structure of non-p-points}, Journal of Symbolic Logic \textbf{81}
  (2016), no.~1, 237--263.

\bibitem{DobrinenJML16}
\bysame, \emph{Infinite dimensional {E}llentuck spaces and
  {R}amsey-classification theorems}, Journal of Mathematical Logic \textbf{16}
  (2016), no.~1, 37 pp.

\bibitem{DobrinenH_317}
\bysame, \emph{The {R}amsey theory of the universal homogeneous triangle-free
  graph},  (2017), 65 pp, Submitted. arXiv:1704.00220v5.

\bibitem{DobrinenSEALS17}
\bysame, \emph{Topological {R}amsey spaces dense in forcings}, Proceedings of
  the 2016 SEALS 2016 Conference (2018), 32 pp, To appear; arXiv:1702.02668v2.

\bibitem{Dobrinen/Hathaway16}
Natasha Dobrinen and Daniel Hathaway, \emph{The {H}alpern-{L}{\"{a}}uchli
  {T}heorem at a measurable cardinal}, Journal of Symbolic Logic \textbf{82}
  (2017), no.~4, 1560--1575.

\bibitem{Dobrinen/Laflamme/Sauer16}
Natasha Dobrinen, Claude Laflamme, and Norbert Sauer, \emph{Rainbow {R}amsey
  simple structures}, Discrete Mathematics \textbf{339} (2016), no.~11,
  2848--2855.

\bibitem{Dobrinen/Mijares/Trujillo14}
Natasha Dobrinen, Jos{\'{e}}~G. Mijares, and Timothy Trujillo,
  \emph{Topological {R}amsey spaces from {F}ra{\"{i}}ss{\'{e}} classes,
  {R}amsey-classification theorems, and initial structures in the {T}ukey types
  of p-points}, Archive for Mathematical Logic, special issue in honor of James
  E. Baumgartner \textbf{56}, no.~7-8, 733--782, (Invited submission).

\bibitem{Dobrinen/Todorcevic14}
Natasha Dobrinen and Stevo Todorcevic, \emph{A new class of
  {R}amsey-classification {T}heorems and their applications in the {T}ukey
  theory of ultrafilters, {P}art 1}, Transactions of the American Mathematical
  Society \textbf{366} (2014), no.~3, 1659--1684.

\bibitem{Dobrinen/Todorcevic15}
\bysame, \emph{A new class of {R}amsey-classification {T}heorems and their
  applications in the {T}ukey theory of ultrafilters, {P}art 2}, Transactions
  of the American Mathematical Society \textbf{367} (2015), no.~7, 4627--4659.

\bibitem{Dzamonja/Larson/MitchellQ09}
M.~D{\v{z}}amonja, J.~Larson, and W.~J. Mitchell, \emph{A partition theorem for
  a large dense linear order}, Israel Journal of Mathematics \textbf{171}
  (2009), 237--284.

\bibitem{Dzamonja/Larson/MitchellRado09}
\bysame, \emph{Partitions of large {R}ado graphs}, Archive for Mathematical
  Logic \textbf{48} (2009), no.~6, 579--606.

\bibitem{Halpern/Lauchli66}
J.~D. Halpern and H.~L{\"{a}}uchli, \emph{A partition theorem}, Transactions of
  the American Mathematical Society \textbf{124} (1966), 360--367.

\bibitem{Halpern/Levy71}
J.~D. Halpern and A.~L{\'{e}}vy, \emph{The {B}oolean prime ideal theorem does
  not imply the axiom of choice}, Axiomatic Set Theory, Proc. Sympos. Pure
  Math., Vol. XIII, Part I, Univ. California, Los Angeles, Calif., 1967,
  American Mathematical Society, 1971, pp.~83--134.

\bibitem{Henson71}
C.~Ward Henson, \emph{A family of countable homogeneous graphs}, Pacific
  Journal of Mathematics \textbf{38} (1971), no.~1, 69--83.

\bibitem{Komjath/Rodl86}
P{\'{e}}ter Komj{\'{a}}th and Vojt{\v{e}}ch R{\"{o}}dl, \emph{Coloring of
  universal graphs}, Graphs and Combinatorics \textbf{2} (1986), no.~1, 55--60.

\bibitem{Laver84}
Richard Laver, \emph{Products of infinitely many perfect trees}, Journal of the
  London Mathematical Society (2) \textbf{29} (1984), no.~3, 385--396.

\bibitem{Milliken81}
Keith~R. Milliken, \emph{A partition theorem for the infinite subtrees of a
  tree}, Transactions of the American Mathematical Society \textbf{263} (1981),
  no.~1, 137--148.

\bibitem{Nesetril/Rodl77}
Jaroslav Ne{\v{s}}et{\v{r}}il and Vojt{\v{e}}ch R{\"{o}}dl, \emph{Partitions of
  finite relational and set systems}, Journal of Combinatorial Theory Series A
  \textbf{22} (1977), no.~3, 289--312.

\bibitem{Nesetril/Rodl83}
\bysame, \emph{Ramsey classes of set systems}, Journal of Combinatorial Theory
  Series A \textbf{34} (1983), no.~2, 183--201.

\bibitem{NVT08}
Lionel Nguyen Van~Th{\'{e}}, \emph{Big {R}amsey degrees and divisibility in
  classes of ultrametric spaces}, Canadian Mathematical Bulletin \textbf{51}
  (2008), no.~3, 413--423.

\bibitem{Sauer98}
Norbert Sauer, \emph{Edge partitions of the countable triangle free homogenous
  graph}, Discrete Mathematics \textbf{185} (1998), no.~1--3, 137--181.

\bibitem{Sauer06}
\bysame, \emph{Coloring subgraphs of the {R}ado graph}, Combinatorica
  \textbf{26} (2006), no.~2, 231--253.

\bibitem{Shelah91}
Saharon Shelah, \emph{Strong partition relations below the power set:
  consistency -- was {S}ierpinski right? ii}, Sets, Graphs and Numbers
  (Budapest, 1991), vol.~60, Colloq. Math. Soc. J{\'{a}}nos Bolyai,
  North-Holland, 1991, pp.~637--688.

\bibitem{TodorcevicBK10}
Stevo Todorcevic, \emph{Introduction to {R}amsey {S}paces}, Princeton
  University Press, 2010.

\bibitem{Farah/TodorcevicBK}
Stevo Todorcevic and Ilijas Farah, \emph{Some {A}pplications of the {M}ethod of
  {F}orcing}, Yenisei Series in Pure and Applied Mathematics, 1995.

\bibitem{Zhang17}
Jing Zhang, \emph{A polarized partition relation for large saturated linear
  orders},  (2017), 23 pp, Submitted, arXiv:1704.06827v2.

\end{thebibliography}

\end{document}